\documentclass[11pt]{amsart}

\newlength{\myhmargin}  
\usepackage[textheight=574pt, textwidth=445pt, marginparwidth=50pt, centering]{geometry}

\usepackage{amsmath,amssymb,amsthm,amsfonts,amscd,flafter,
graphicx,verbatim,pinlabel,mathrsfs}
\usepackage[all]{xy}
\usepackage{epstopdf}
\usepackage[colorlinks=false]{hyperref}
\usepackage[all]{hypcap}
\AtBeginDocument{\addtocontents{toc}{\protect\setlength{\parskip}{0pt}}}

\title{Invariants of Legendrian and transverse knots in monopole knot homology}

\author[John A. Baldwin]{John A. Baldwin}
\address{Department of Mathematics \\ Boston College}
\email{john.baldwin@bc.edu}

\author[Steven Sivek]{Steven Sivek}
\address{Department of Mathematics \\ Princeton University}
\email{ssivek@math.princeton.edu}

\thanks{JAB was partially supported by NSF grant DMS-1104688.  SS was partially supported by NSF postdoctoral fellowship DMS-1204387.}

\def\R{{\mathbb{R}}}
\def\C{{\mathbb{C}}}

\newcommand\zz{\mathbb{Z}}

\newcommand\Sc{\text{Spin}^c}
\newcommand\spc{\mathfrak{s}}

\newcommand\ssm{\smallsetminus}

\newcommand\Psit{\underline{\Psi}}
\newcommand\Z{\mathbb{Z}}

\newcommand\inr{{\rm int}}
\newcommand\RR{\mathcal{R}}

\newcommand\data{\mathscr{D}}
\newcommand\SFH{SFH}
\newcommand\SHM{SHM}
\newcommand\SHMt{\underline{\SHM}}
\newcommand\KHM{KHM}
\newcommand\KHMt{\underline{\KHM}}

\newcommand{\definefunctor}[1]{\textbf{\textup{#1}}}

\newcommand\SHMtfun{\definefunctor{\SHMt}}

\newcommand\KHMtfun{\definefunctor{\KHMt}}

\newcommand\Img{{\rm Im}}
\newcommand\invt{\psi} 
\newcommand\kinvt{\mathcal{L}} 
\newcommand\tinvt{\mathcal{T}} 
\newcommand\HMtoc{\HMto_{\bullet}}

\newcommand\PSys{\textbf{\textup{PSys}}}
\newcommand{\RPSys}[1][\RR]{{#1}\mbox{-}\PSys}
\newcommand\DiffSut{\textup{\textbf{DiffSut}}}

\newcommand{\longcomment}[2]{#2}

\DeclareFontFamily{U}{mathx}{\hyphenchar\font45}
\DeclareFontShape{U}{mathx}{m}{n}{
      <5> <6> <7> <8> <9> <10>
      <10.95> <12> <14.4> <17.28> <20.74> <24.88>
      mathx10
      }{}
\DeclareSymbolFont{mathx}{U}{mathx}{m}{n}
\DeclareFontSubstitution{U}{mathx}{m}{n}
\DeclareMathAccent{\widecheck}{0}{mathx}{"71}
\newcommand{\HMto}{\widecheck{\mathit{HM}}}

\longcomment{
    \RequirePackage{rotating}                   
    \def\HMto{%
       \setbox0=\hbox{$\widehat{\mathit{HM}}$}
       \setbox1=\hbox{$\mathit{HM}$}
       \dimen0=1.1\ht0
       \advance\dimen0 by 1.17\ht1
       \smash{\mskip2mu\raise\dimen0\rlap{%
          \begin{turn}{180}
              {$\widehat{\phantom{\mathit{HM}}}$}
           \end{turn}} \mskip-2mu    
                \mathit{HM}
    }{\vphantom{\widehat{\mathit{HM}}}}{}}
}
    
    \newcommand*\oline[1]{%
  \vbox{%
    \hrule height 0.35pt
    \kern0.1ex
    \hbox{%
      \kern-0.0em
      \ifmmode#1\else\ensuremath{#1}\fi
      \kern-0.1em
    }
  }
}

\newtheorem{theorem}{Theorem}[section]
\newtheorem{lemma}[theorem]{Lemma}

\newtheorem{conjecture}[theorem]{Conjecture}
\newtheorem{corollary}[theorem]{Corollary}
\newtheorem{proposition}[theorem]{Proposition}

\theoremstyle{definition}
\newtheorem{definition}[theorem]{Definition}
\newtheorem{notation}[theorem]{Notation}

\newtheorem{remark}[theorem]{Remark}

\newtheorem{example}[theorem]{Example}

\makeatletter
\newtheorem*{rep@thm}{\rep@title}
\newcommand{\newreptheorem}[2]{%
\newenvironment{rep#1}[1][0,0]{%
\def\rep@title{#2##1}%
\begin{rep@thm}}%
{\end{rep@thm}}}
\makeatother
\newreptheorem{theorem}{}

\begin{document}
\begin{abstract} 
We use the contact invariant defined in  \cite{bsSHM} to construct a new invariant of Legendrian knots in Kronheimer and Mrowka's monopole knot homology theory ($\KHM$), following a prescription of Stipsicz and V{\'e}rtesi. Our Legendrian invariant improves upon an earlier Legendrian invariant in $\KHM$ defined by the second author in several important respects. Most notably, ours is preserved by negative stabilization. This fact   enables us to define a transverse knot invariant in $\KHM$ via Legendrian approximation. It also  makes our invariant a more likely candidate for the monopole Floer analogue of the ``LOSS" invariant in knot Floer homology. Like its predecessor, our Legendrian invariant behaves functorially with respect to Lagrangian concordance. We show how this fact can be used  to compute our invariant in several  examples.
As a byproduct of our investigations, we provide the first infinite family of nonreversible Lagrangian concordances between prime knots.


 \end{abstract}

\maketitle

\section{Introduction}
\label{sec:intro}
A basic goal in contact geometry is to construct invariants that can distinguish Legendrian or transverse knots in a contact 3-manifold which are smoothly isotopic and have the same classical invariants but are not Legendrian or transversely isotopic. Such an invariant is said to be \emph{effective}. Effective Legendrian invariants include Chekanov and Eliashberg's Legendrian contact homology ($LCH$) \cite{yasha8,ehk} and the ``LOSS" and ``GRID" invariants in knot Floer homology defined by Lisca, Ozsv{\'a}th, Stipsicz, and Szab{\'o} \cite{lossz} and Ozsv{\'a}th, Szab{\'o}, and Thurston \cite{oszt}, respectively. Effective transverse invariants are  harder to come by; the only known examples are those arising from the LOSS and GRID invariants via Legendrian approximation and the transverse knot contact homology theory developed by Ng et al. \cite{eens, ngtransverse}.

In this paper, we define a new invariant of   Legendrian knots using Kronheimer and Mrowka's monopole knot homology theory ($KHM$). This theory is the monopole Floer analogue of knot Floer homology and we expect that our invariant is the corresponding   analogue of the LOSS invariant in a sense made precise later. 
It bears mentioning that ours is the second Legendrian invariant defined in $\KHM$; the first was defined by the second author in \cite{sivek}. However, our construction is substantially different and  improves upon this earlier invariant in several important respects. Most notably, our invariant is preserved by negative stabilization and thus gives rise to a new   transverse invariant  via Legendrian approximation. 

Like the second author's invariant, ours behaves functorially with respect to Lagrangian concordance, something which is not known to be true of the LOSS invariant. If the LOSS invariant \emph{were} functorial in this way, then the equivalence  \cite{bvv} between the LOSS and GRID invariants  would provide easily computable obstructions to the existence of Lagrangian concordances between Legendrian knots in the tight contact structure on $S^3$. Indeed, one of our future goals is to use the construction in this paper to prove this functoriality for the LOSS invariant, as explained in more detail at the end of this introduction.

Below, we outline the constructions of our Legendrian and transverse invariants  and describe some of their properties, elaborating on several points in the discussion above. We also describe how our investigations led us to an infinite family of nonreversible Lagrangian concordances between prime knots. At the end, we discuss plans for future work.

\subsection{The contact invariant in sutured monopole homology} Our Legendrian and transverse invariants are defined in terms of the contact invariant in sutured monopole homology ($SHM$) constructed in \cite{bsSHM}. We  therefore recall this construction briefly  below.

The \emph{sutured monopole homology}  of a balanced sutured manifold $(M,\Gamma)$, as defined by Kronheimer and Mrowka in \cite{km4},  is an isomorphism class of $\RR$-modules, denoted by $\SHMt(M,\Gamma)$, where $\RR$ is the Novikov ring with integer coefficients. It is defined in terms of the monopole Floer homology of a \emph{closure} of $(M,\Gamma)$, which is a pair $(Y,R)$, where $Y$ is a certain closed 3-manifold containing $M$ and  $R$ is a distinguished surface in $Y$. In \cite{bs3}, we introduced canonical isomorphisms, well-defined up to multiplication by units in $\RR$, relating the Floer  homology groups associated to different closures of $(M,\Gamma)$. These  Floer  groups and isomorphisms make up  what we call a \emph{projectively transitive systems of $\RR$-modules}, denoted by $\SHMtfun(M,\Gamma)$. This is  a more subtle   invariant of $(M,\Gamma)$ than the isomorphism class $\SHMt(M,\Gamma)$. In particular, it makes sense to talk about elements of and morphisms between projectively transitive systems, whereas these notions are much less interesting  for isomorphism classes of $\RR$-modules.

Suppose   $\xi$ is a contact structure on $M$ such that $\partial M$ is convex with dividing set $\Gamma$. In \cite{bsSHM}, we introduced the notion of a \emph{contact closure} of the \emph{sutured contact manifold} $(M,\Gamma,\xi)$. Roughly, this is a triple $(Y,R,\bar\xi)$, where $(Y,R)$ is a closure of $(M,\Gamma)$ and $\bar\xi$ is a certain contact structure on $Y$ extending $\xi$. For each such contact closure, the contact invariant $\psi(Y,\bar\xi)$ in the monopole Floer homology of $-Y$, as defined by Kronheimer and Mrowka in \cite{km, kmosz},  determines an element of $\SHMtfun(-M,-\Gamma)$. We proved that for contact closures of sufficiently high genus (where the \emph{genus} of a closure refers to that of its distinguished surface $R$), the induced elements of $\SHMtfun(-M,-\Gamma)$   all agree---that is, they are independent of the closure. Our contact invariant is defined to be this common element, denoted by \[\invt(M,\Gamma,\xi)\in \SHMtfun(-M,-\Gamma).\]

\subsection{Legendrian and transverse invariants in monopole knot homology} 
\label{ssec:introleg} The \emph{monopole knot homology} of a knot $K$ in a closed 3-manifold $Y$, as defined by Kronheimer and Mrowka in \cite{km4}, is the  isomorphism class of $\RR$-modules \[\KHMt(Y,K):=\SHMt(Y\ssm \nu(K),m \cup -m),\] where $\nu(K)$ is a tubular neighborhood of the knot and $m$ is an oriented meridian on the boundary of this knot complement. In \cite{bs3}, we introduced a refinement of this invariant which assigns to a based knot $(K,p)$ in $ Y$ a projectively transitive system of $\RR$-modules,  denoted by $\KHMtfun(Y,K,p)$, which is similarly defined in terms of the systems $\SHMtfun(Y\ssm \nu(K),m \cup -m)$ associated to the various possible knot complements $(Y\ssm \nu(K),m \cup -m)$.

The Legendrian invariant defined in this paper assigns to a based, oriented Legendrian knot $(K,p)$ in $(Y,\xi)$ an element \[\kinvt(K)\in\KHMtfun(-Y,K,p).\] To define this element, we first  remove a standard neighborhood of $K$ and then glue on a piece called a \emph{bypass} (roughly, half of a thickened overtwisted disk) in such a way that the result is the complement of a tubular neighborhood of $K$ with dividing set  a pair of  oppositely oriented meridians. Suppose $(Y\ssm \nu(K),m\cup -m, \xi_K)$ is the sutured contact manifold formed in this way. The Legendrian invariant of $K$ is then  defined, very roughly speaking, by \[\kinvt(K):=\invt(Y\ssm \nu(K),m\cup -m, \xi_K)\in\KHMtfun(-Y,K,p).\] We prove that, up to isomorphism, this class is preserved by contactomorphism and Legendrian isotopy (see Proposition \ref{prop:Lfunct} and Corollary \ref{cor:Liso} for more precise statements). In particular, a  Legendrian isotopy $f_t$ sending $(K,p)$ to $(K',p')$ gives rise to a well-defined map \[\Psi_{f_t}:\KHMtfun(-Y,K,p)\to\KHMtfun(-Y,K',p')\] sending $\kinvt(K)$ to $\kinvt(K')$. In \cite{ost}, Ozsv{\'a}th and Stipsicz used a similar sort of naturality statement  about the LOSS invariant (which is true modulo incorporating the results of \cite{juhaszthurston}) to distinguish smoothly isotopic Legendrian knots with the same classical invariants.

Our construction is inspired by work of Stipsicz and V{\'e}rtesi \cite{sv} who proved that the LOSS invariant can be formulated in a similar way in terms of Honda, Kazez, and Mati{\'c}'s contact invariant in sutured (Heegaard) Floer homology \cite{hkm4}. We also use their work to prove that our invariant shares some important features with the LOSS invariant---most notably,  that $\kinvt$ is  preserved by negative Legendrian stabilization (Theorem \ref{thm:natstab}). This fact allows us to define an invariant of based, oriented transverse knots via Legendrian approximation as in \cite{lossz,ost}. Namely, given a transverse knot $(K,p) \subset (Y,\xi)$, we choose a Legendrian pushoff $(K',p')$ of $(K,p)$ in a standard neighborhood of $K$, and define the transverse invariant of $K$ to be \[\tinvt(K):=\kinvt(K')\in\KHMtfun(-Y,K,p).\footnote{After  identifying $\SHMtfun(-Y,K',p')$ with $\SHMtfun(-Y,K,p)$ in a canonical way.} \] The fact that any two such Legendrian pushoffs of $K$ are related by negative stabilization and Legendrian isotopy \cite{efm} implies that $\tinvt(K)$ is well-defined and is preserved, up to  natural isomorphism, by contactomorphism and  transverse isotopy (Proposition \ref{prop:Tfunct} and Corollary \ref{cor:Tiso}); in particular, transverse isotopies give rise to well-defined maps sending transverse invariant to transverse invariant as discussed above for the Legendrian invariant.

We expect the following  to be true.

\begin{conjecture}
\label{conj:effective}
$\kinvt$ and $\tinvt$ are \emph{effective} invariants of Legendrian and transverse knots.
\end{conjecture}

As mentioned at the beginning of this introduction, ours is the second Legendrian invariant constructed using monopole knot homology. In \cite{sivek}, the second author defined an invariant  which associates to a Legendrian knot $K\subset (Y,\xi)$ a sequence of ``elements" \[\ell(K) = \{\ell_g(K)\in \KHMt(-Y,K)\}_{g=2}^{\infty}.\] Each $\ell_g(K)$ is defined in terms of Kronheimer and Mrowka's contact invariant for closed contact 3-manifolds,  after extending the contact structure on the complement of a standard neighborhood of $K$  to a  contact structure on a genus $g$ closure of the  knot complement with its meridional sutures. This invariant was also designed as a potential monopole Floer analogue of the LOSS invariant, and is also conjectured to be an effective invariant of Legendrian knots.

As mentioned above, our Legendrian invariant $\kinvt$ improves upon $\ell$ in a few important respects. First, while our invariant is also fundamentally a sequence $\{\kinvt_g\}$ of elements, one for each genus,  we are able to show that all of these $\kinvt_g$ are  equal for $g$ large enough, whereas the corresponding statement for the $\ell_g$ is only conjectured in \cite{sivek}. Second, the fact that our invariant is defined in terms of our ``natural" refinement of monopole knot homology makes it possible in principle to distinguish  the  invariants associated to two Legendrian representatives of the same based knot type, even when both are nonzero, something that is not possible for nonzero ``elements" of an isomorphism class of $\RR$-modules. Finally, as discussed above, we prove that $\kinvt$ is preserved by negative stabilization and can therefore be used to define a transverse invariant. By contrast, the second author shows in \cite{sivek} that $\ell$ is killed by a composition of positive and negative Legendrian stabilizations, but is not able to prove that it is preserved by negative stabilization. In fact, as defined, $\ell$ is an invariant of \emph{unoriented} Legendrian knots and so  \emph{cannot} be used to distinguish between positive and negative stabilization (or to define an transverse invariant).

One of the most interesting features of our Legendrian invariant  is that, like $\ell$,  it is functorial with respect to Lagrangian concordance in the following sense (Theorem \ref{thm:lagconcordance}).

\begin{theorem}
\label{thm:introlagconcordance}
Suppose $(K_-,p_-)$ and $(K_+,p_+)$ are Legendrian knots in $(Y,\xi)$ such that   $K_-$ is Lagrangian concordant to $K_+$ in the symplectization of $(Y,\xi)$. Then there exists a map \[\KHMtfun(-Y,K_+,p_+)\to\KHMtfun(-Y,K_-,p_-)\] which sends $\kinvt(K_+)$ to $\kinvt(K_-)$.
\end{theorem}

We expect this map to be an invariant of the Lagrangian cylinder from $K_-$ to $K_+$, perhaps after decorating this cylinder with an arc from $p_-$ to $p_+$, but leave this for future work. 

Our proof of Theorem \ref{thm:introlagconcordance} is modeled on the proof of the analogous result for $\ell$:  roughly, we remove a standard neighborhood of the Lagrangian concordance in $Y\times \R$ and glue back the symplectization of a certain contact manifold with boundary, so as to form an exact symplectic cobordism between the contact closures used to define  $\kinvt(K_{\pm})$. The theorem then follows from the functoriality of Kronheimer and Mrowka's contact invariant for closed contact 3-manifolds under exact symplectic cobordism. It is not known whether Ozsv{\'a}th and Szab{\'o}'s contact invariant in Heegaard Floer homology enjoys the same kind of functoriality, which makes proving the analogue of Theorem \ref{thm:introlagconcordance} for the LOSS invariant difficult. 

In Section \ref{sec:lagcobordism}, we use Theorem \ref{thm:introlagconcordance} to prove the nonvanishing of  $\kinvt(K)$ for a few examples, including Legendrian representatives of $m(8_{20})$ and $m(11n_{71})$. Our approach  is  to find another  knot $K'$ with nonzero invariant (for example, the $tb=-1$ unknot) and construct  a Lagrangian concordance from some negative stabilization of $K'$ to $K$ (stabilizing gives us more flexibility in constructing the concordance). The nonvanishing of $\kinvt(K)$ then follows from properties of $\kinvt$ discussed above.

Legendrian contact homology   enjoys a similar functoriality under Lagrangian cobordism. However, one cannot use the above  technique to deduce  nonvanishing results for $LCH$ in our examples since   $LCH$ vanishes for stabilized knots. On the other hand,  $LCH$ \emph{can} be used to show that the concordances  we construct are \emph{nonreversible}. The first examples of nonreversible Lagrangian concordances were discovered by Chantraine  \cite{chantraine}. In the course of our investigations, we found a new infinite family of nonreversible concordances, from Legendrian unknots to 3-stranded Legendrian pretzel knots. While this discovery  does not rely on our Legendrian invariant, it is perhaps of independent interest as it provides the first infinite family of nonreversible Lagrangian concordances between \emph{prime} knots.

\subsection{Future work} One of our major goals for future work is to prove an equivalence between our Legendrian and transverse knot invariants and the LOSS/GRID invariants defined in knot Floer homology, as explained below.

 Suppose $(M,\Gamma,\xi)$ is a sutured contact manifold. The combined work of Kronheimer and Mrowka \cite[Lemma 4.9]{km4}, Taubes \cite{taubes1,taubes2,taubes3,taubes4,taubes5}, Colin, Ghiggini, and Honda \cite{cgh3, cgh4, cgh5}, and Lekili \cite{lekili2} shows that \begin{equation}\label{eqn:isomf}\SHMtfun(-M,-\Gamma) \cong \SFH(-M,-\Gamma)\otimes\RR.\footnote{One can think of  $\SFH(-M,-\Gamma)\otimes\RR$ as a projectively transitive system. See also Kutluhan, Lee, and Taubes \cite{klt1,klt2,klt3,klt4,klt5}. }\end{equation} In the introduction to \cite{bsSHM}, we discuss a strategy for showing that this isomorphism identifies our contact invariant $\invt(M,\Gamma,\xi)$ with Honda, Kazez, and Mati{\'c}'s contact invariant \[EH(M,\Gamma,\xi)\otimes 1\in \SFH(-M,-\Gamma)\otimes \RR.\] Now suppose $(K,p)$ is a based, oriented Legendrian knot in $(Y,\xi)$. A special case of  (\ref{eqn:isomf}) is the following isomorphism between monopole knot homology and knot Floer homology, \begin{equation}\label{eqn:isomfk}\KHMtfun(-Y,K,p) \cong \widehat{HFK}(-Y,K,p)\otimes\RR.\end{equation}
By our construction of $\kinvt$, combined with Stipsicz and V{\'e}rtesi's interpretation of the LOSS invariant, the proposed equivalence of our contact invariant with Honda, Kazez, and Mati{\'c}'s would imply that the isomorphism in (\ref{eqn:isomfk}) identifies $\kinvt(K)$ with the LOSS invariant of $K$.

Such an identification would imply that $\kinvt$ and $\tinvt$ are effective, proving Conjecture \ref{conj:effective}. It would also establish a  functoriality result for the LOSS invariant with respect to Lagrangian concordance, analogous to Theorem \ref{thm:introlagconcordance}. As explained above, this functoriality, combined with the equivalence of the GRID and LOSS invariants for Legendrian knots in the standard tight $S^3$ would provide easily computable obstructions to the existence of Lagrangian concordances between such  knots. Interestingly, such an identification would also provide a proof of invariance of the LOSS invariant that does not depend on the hard/controversial ``uniqueness" part of the relative Giroux correspondence (see the discussion in the introduction to \cite{bsSHM}).

\subsection{Organization} 

In Section \ref{sec:prelims}, we  review  projectively transitive systems, the constructions of sutured monopole homology, monopole knot homology, and our contact invariant in $\SHM$. In Section \ref{sec:leginvt}, we define the  Legendrian and transverse knot invariants $\kinvt$ and $\tinvt$ and establish some of their basic properties. In Section \ref{sec:lagcobordism}, we prove that our Legendrian invariant is functorial with respect to Lagrangian concordance. In Section \ref{sec:examples}, we illustrate how this functoriality may be used to compute $\kinvt$ in several examples,  and we describe an infinite family of nonreversible Lagrangian concordances between prime knots.

\subsection{Acknowledgements} We thank Lenny Ng and Danny Ruberman for helpful conversations.

\section{Preliminaries}
\label{sec:prelims}

In this section, we review projectively transitive systems and the constructions of sutured monopole homology, monopole knot homology, and  our contact invariant from \cite{bsSHM}.

\subsection{Projectively transitive systems}
\label{ssec:transitive-systems}
The following is a very terse review of \emph{projectively transitive systems}; see \cite{bs3,bsSHM} for more details. 

\begin{definition}
Suppose $M_\alpha$ and $M_\beta$ are modules over a unital commutative ring $\RR$. We say that  elements $x,y\in M_\alpha$ are \emph{equivalent} if $x=u\cdot y$ for some  $u\in \RR^\times$. Likewise,  homomorphisms \[f,g: M_\alpha \to M_\beta\]  are  \emph{equivalent} if $f=u\cdot g$ for some  $u\in \RR^\times$.  
\end{definition}

\begin{remark}
We will write $x\doteq y$ or $f\doteq g$ to indicate that two  elements or homomorphisms are equivalent, and  will denote their equivalence classes by $[x]$ or $[f]$.
\end{remark}

\begin{definition}
Let $\RR$ be a unital commutative ring.  A \emph{projectively transitive system of $\RR$-modules} consists of a set $A$ together with:
\begin{enumerate}
\item a collection of $\RR$-modules $\{M_\alpha\}_{\alpha \in A}$ and
\item a collection of equivalence classes of homomorphisms $\{g^\alpha_\beta\}_{\alpha,\beta \in A}$ such that:
\begin{enumerate}
\item elements of the equivalence class $g^\alpha_\beta$ are isomorphisms from $M_\alpha$ to $M_\beta$, 
\item  $g^\alpha_\alpha=[id_{M_\alpha}]$,
\item $g^\alpha_\gamma = g^\beta_\gamma \circ g^\alpha_\beta$.
\end{enumerate}
\end{enumerate}
\end{definition}


The class of  projectively transitive systems of $\RR$-modules forms a category $\RPSys{}$ with the following notion of morphism.

\begin{definition}
\label{def:projtransysmor} A \emph{morphism} of  projectively transitive systems of $\RR$-modules \[F:(A,\{M_{\alpha}\},\{g^{\alpha}_{\beta}\})\to(B,\{N_{\gamma}\},\{h^{\gamma}_{\delta}\})\]  is  a collection of equivalence classes  of homomorphisms $F=\{F^\alpha_\gamma\}_{\alpha\in A,\,\gamma\in B}$  such that:
\begin{enumerate}
\item elements of the equivalence class $F^\alpha_\gamma$ are homomorphisms from $M_\alpha$ to $N_\gamma$,
\item \label{eqn:projtransysmor} $F^\beta_\delta\circ g^\alpha_\beta = h^\gamma_\delta\circ F^\alpha_\gamma$.
\end{enumerate}  
Note that $F$ is an \emph{isomorphism} iff the elements in each equivalence class $F^\alpha_\gamma$ are isomorphisms.
\end{definition}

\begin{remark}\label{rmk:completesubset} A collection of equivalence classes of homomorphisms $\{F^\alpha_\gamma\}$  with indices  ranging over any nonempty subset of $A\times B$ can be uniquely completed to a morphism as long as this collection satisfies the compatibility in \eqref{eqn:projtransysmor} where it makes sense.
\end{remark}

\begin{definition}
\label{def:element}
 An \emph{element} of a projectively transitive system of $\RR$-modules \[x\in\mathcal{M}=(A,\{M_{\alpha}\},\{g^{\alpha}_{\beta}\})\]
 is a collection of equivalence classes of elements $x = \{x_{\alpha}\}_{\alpha\in A}$ such that:
\begin{enumerate}
\item elements of the equivalence class $x_{\alpha}$ are elements of $M_\alpha$,
\item \label{eqn:projtransysmorelt} $x_\beta = g^\alpha_\beta(x_\alpha)$.
\end{enumerate}  
\end{definition}

\begin{remark}
\label{rmk:completesubset2}
 As in Remark \ref{rmk:completesubset}, a collection of equivalence classes of elements $\{x_{\alpha}\}$ with indices ranging over any nonempty subset of $A$ can be uniquely completed to an element of $\mathcal{M}$ as long as this collection satisfies the compatibility in \eqref{eqn:projtransysmorelt} where it makes sense.
 \end{remark}

\subsection{Sutured monopole homology}
\label{ssec:shm}

In this subsection, we review  our refinement of Kronheimer and Mrowka's sutured monopole homology, following \cite{bs3}.

Suppose $(M,\Gamma)$ is a \emph{balanced sutured manifold} as in  \cite[Definition 2.10]{bsSHM}.
An \emph{auxiliary surface} for $(M,\Gamma)$ is a compact, connected, oriented surface $F$ with $g(F)>0$ and $\pi_0(\partial F)\cong \pi_0(\Gamma)$.  Given such a surface, a closed tubular neighborhood $A(\Gamma)$  of $\Gamma$ in $\partial M$, and  an orientation-reversing diffeomorphism  \[h:\partial F\times[-1,1]\rightarrow A(\Gamma)\] which sends $\partial F\times \{\pm 1\}$ to $\partial (R_{\pm}(\Gamma)\smallsetminus A(\Gamma)),$ we form a \emph{preclosure}  \begin{equation*}\label{eqn:bF}M'=M\cup_h F\times [-1,1]\end{equation*}  by gluing $F\times[-1,1]$ to $M$ according to $h$ and rounding corners.  The \emph{balanced} condition on $(M,\Gamma)$ ensures that $M'$ has two diffeomorphic boundary components, $\partial_+ M'$ and $\partial_- M'$, which we may glue together by some diffeomorphism   to form a closed 3-manifold $Y$ containing a distinguished surface \[R:=\partial_+M' = -\partial_- M'\subset Y.\] In \cite{km4}, Kronheimer and Mrowka define a \emph{closure} of $(M,\Gamma)$ to be any pair $(Y,R)$ obtained in this way.  Our definition of closure, as needed for naturality, is  slightly more involved.

\begin{definition}[\cite{bs3}]
\label{def:smoothclosure} A \emph{marked closure} of $(M,\Gamma)$ is a tuple $\data = (Y,R,r,m,\eta)$ consisting of:
\begin{enumerate}
\item a closed, oriented,  3-manifold $Y$,
\item  a closed, oriented,  surface $R$ with $g(R)\geq 2$,
\item an oriented, nonseparating, embedded curve $\eta\subset R$,
\item a smooth, orientation-preserving embedding $r:R\times[-1,1]\hookrightarrow Y$,
\item a smooth, orientation-preserving embedding $m:M\hookrightarrow Y\smallsetminus\inr(\Img(r))$ such that: 
\begin{enumerate}
\item $m$ extends  to a diffeomorphism \[M\cup_h F\times [-1,1]\rightarrow Y\smallsetminus{\rm int}(\Img(r))\] for some $A(\Gamma)$, $F$, $h$, as above,
\item $m$ restricts to an orientation-preserving embedding \[R_+(\Gamma)\smallsetminus A(\Gamma)\hookrightarrow r(R\times\{-1\}).\]
\end{enumerate}
 \end{enumerate} 
 The \emph{genus} $g(\data)$ refers to the genus of $R$.
\end{definition}


 \begin{remark} Suppose $\data = (Y,R,r,m,\eta)$ is a marked closure of $(M,\Gamma)$. Then, the tuple \[-\data:=(-Y,-R,r,m,-\eta),\] obtained  by reversing the orientations of $Y$, $R$, and $\eta$, is a marked closure of $-(M,\Gamma):=(-M,-\Gamma),$ where $r$ and $m$ are the induced embeddings  of $-R\times[-1,1]$ and $-M$ into $-Y$. 
  \end{remark}
 
 

\begin{notation}For the rest of this paper,  $\RR$ will be the \emph{Novikov ring over $\zz$}, given by \[\RR=\bigg\{\sum_{\alpha}c_{\alpha}t^{\alpha}\,\bigg | \,\alpha\in\mathbb{R},\,c_{\alpha}\in\zz,\,\#\{\beta<n|c_{\beta}\neq 0\}<\infty\textrm{ for all } n\in \zz\bigg\}.\] 
\end{notation}

Following Kronheimer and Mrowka \cite{km4}, we made the following definition in \cite{bs3}.

\begin{definition}\label{def:shmt}Given a marked closure $\data = (Y,R,r,m,\eta)$ of $(M,\Gamma)$, the \emph{sutured monopole homology of $\data$} is the $\RR$-module \[\SHMt(\data):=\HMtoc(Y|R;\Gamma_\eta).\]
\end{definition}

Here, $\HMtoc(Y|R;\Gamma_\eta)$ is shorthand for the monopole Floer homology of $Y$ in the ``topmost" $\Sc$ structures relative to $r(R\times\{0\})$, 
\begin{equation}\label{eqn:relativelocal}\HMtoc(Y|R;\Gamma_\eta):=\bigoplus_{\substack{\spc\in\Sc(Y)\\ \langle c_1(\spc), [r(R\times\{0\})]\rangle = 2g(R)-2}} \HMtoc(Y,\spc; \Gamma_{r(\eta\times\{0\})}), \end{equation} where, for each $\Sc$ structure $\mathfrak{s}$,  $\Gamma_{r(\eta\times\{0\})}$ is the local system on the Seiberg-Witten configuration space $\mathcal{B}(Y,\mathfrak{s})$ with fiber  $\RR$ specified by  $r(\eta\times\{0\})\subset Y$ as  in \cite[Section 2.2]{km4}.

In \cite{km4}, Kronheimer and Mrowka proved that the isomorphism class of $\SHMt(\data)$ is an invariant of $(M,\Gamma)$. We strengthened this in \cite{bs3}: for any two marked closures $\data,\data'$ of $(M,\Gamma)$, we constructed an isomorphism \[\Psit_{\data,\data'}:\SHMt(\data)\to\SHMt(\data'),\]  well-defined up to multiplication by a unit in $\RR$, such that the modules in $\{\SHMt(\data)\}_{\data}$ and the equivalence classes of maps in $\{\Psit_{\data,\data'}\}_{\data,\data'}$ form a projectively transitive system of $\RR$-modules.

\begin{definition}
The \emph{sutured monopole homology of $(M,\Gamma)$} is the projectively transitive system of $\RR$-modules $\SHMtfun(M,\Gamma)$ given by the modules  and the equivalence classes above.
\end{definition} 

Sutured monopole homology is functorial in the following sense. Suppose \[f:(M,\Gamma)\to(M',\Gamma')\] is a diffeomorphism of sutured manifolds and $\data' = (Y',R',r',m',\eta')$ is a marked closure of $(M',\Gamma')$. Then \begin{equation}\label{eqn:dataf}\data'_f:=(Y',R',r',m'\circ f,\eta')\end{equation} is a marked closure of $(M,\Gamma)$. Let \[id_{\data'_f,\data'}: \SHMt(\data'_f)\to\SHMt(\data')\] be the identity map on $\SHMt(\data'_f) = \SHMt(\data')$. The equivalence classes of these identity maps can be completed to a morphism (as in Remark \ref{rmk:completesubset}) \[\SHMtfun(f):\SHMtfun(M,\Gamma)\to\SHMtfun(M',\Gamma'),\]  which is an invariant of the isotopy class of $f$. We proved in \cite{bs3} that these morphisms behave as expected under composition of diffeomorphisms, so that  $\SHMtfun$ defines a functor from $\DiffSut$ to $\RPSys{},$ where $\DiffSut$ is the category of balanced sutured manifolds and isotopy classes of diffeomorphisms between them.

\subsection{Monopole knot homology} 
\label{ssec:khm} In this subsection, we review our refinement of Kronheimer and Mrowka's monopole knot homology, following \cite{bs3}.

Suppose $K$ is an oriented knot in a closed 3-manifold $Y$ and $p$ is a basepoint on $K$. Let $D^2$ be the unit disk in the complex plane and let $S^1 = \partial D^2$. Suppose \[\varphi:S^1\times D^2\to Y\] is an embedding such that $\varphi(S^1\times\{0\})=K$ and $\varphi(\{1\}\times\{0\}) = p$. Let $Y(\varphi)$ be the balanced sutured manifold given by \[Y(\varphi):=(Y\ssm \inr(\Img(\varphi)), m^+_{\varphi}\cup -m^-_{\varphi}),\] where $m^{\pm}_{\varphi}$ are the oriented meridians on $\partial Y(\varphi)$ given by \[m^\pm_{\varphi}:=\varphi(\{\pm 1\}\times \partial D^2).\] The monopole knot homology  $\KHMtfun(Y,K,p)$ is defined, to first approximation,  as $\SHMtfun(Y(\varphi))$. Of course, this does not technically make sense since the latter  depends on $\varphi$. However, in \cite{bs3}, we defined a canonical isomorphism \[\Psit_{\varphi,\varphi'}:\SHMtfun(Y(\varphi))\to\SHMtfun(Y(\varphi'))\] for any two embeddings $\varphi,\varphi'$ as above.  These maps allow us to define $\KHMtfun(Y,K,p)$ without ambiguity (see Definition \ref{def:khm}). The map $\Psit_{\varphi,\varphi'}$ is constructed in two steps, as described below.

We first consider the case  in which $\Img(\varphi')\subset \Img(\varphi)$. 
Let $N$ be a solid torus neighborhood of $K$ with $\Img(\varphi)\subset \inr(N)$. Let $f_t:Y\to Y$, $t\in[0,1]$, be an ambient isotopy such that:
\begin{enumerate}
\item each $f_t$ fixes $p$,
\item each $f_t$ restricts to the identity outside of $N$,
\item $\Img(f_1\circ\varphi)=\Img(\varphi')$,
\item $f_1$ sends the meridional disks $\varphi(\{\pm 1\} \times D^2)$ to the meridional disks $\varphi'(\{\pm 1\}\times  D^2)$.
\end{enumerate} 
Then $f_1$ restricts to a diffeomorphism of sutured manifolds, \[\bar f_1:Y(\varphi)\to Y(\varphi'),\]  and we define $\Psit_{\varphi,\varphi'}$ by \begin{equation*}\label{eqn:embeddingscon}\Psit_{\varphi,\varphi'}:=\SHMtfun(\bar f_1).\end{equation*}
Next, we consider the case in which  $\varphi,\varphi'$ are arbitrary embeddings. Let $\varphi''$ be a third embedding with $\Img(\varphi'')\subset \Img(\varphi)\cap \Img(\varphi')$. We define 
\begin{equation*}\label{eqn:embeddingsgen}\Psit_{\varphi,\varphi'}:=(\Psit_{\varphi',\varphi''})^{-1}\circ \Psit_{\varphi,\varphi''},\end{equation*} where the maps $\Psit_{\varphi',\varphi''}$ and  $\Psit_{\varphi,\varphi''}$ are  as defined previously. 

We proved in \cite{bs3} that these maps are well-defined  and satisfy the transitivity relation \[\Psit_{\varphi,\varphi''} = \Psit_{\varphi',\varphi''}\circ\Psit_{\varphi,\varphi'}.\] The projectively transitive systems in $\{\SHMtfun(Y(\varphi))\}_{\varphi}$ and the isomorphisms in $\{\Psit_{\varphi,\varphi'}\}_{\varphi,\varphi'}$ thus form  a \emph{transitive system of projectively transitive systems of $\RR$-modules}, which determines a larger projectively transitive system of $\RR$-modules, leading to the following  from \cite{bs3}.

\begin{definition}
\label{def:khm}
The monopole knot homology $\KHMtfun(Y,K,p)$ is the projectively transitive system of $\RR$-modules determined by $\{\SHMtfun(Y(\varphi))\}_{\varphi}$ and $\{\Psit_{\varphi,\varphi'}\}_{\varphi,\varphi'}$.
\end{definition}  

When we do not wish to keep track of $p$, we will simply use the notation $\KHMtfun(Y,K)$.

Monopole knot homology is functorial in the following sense. Suppose $f$ is a diffeomorphism from $(Y,K,p)$ to $(Y',K',p')$. For each tubular neighborhood  $\varphi$ of $K$ as  defined above,  $f$ defines a diffeomorphism of balanced sutured manifolds,  \[\bar f:Y(\varphi)\to Y'(f\circ\varphi).\]  The map $\bar f$  then induces a map \[\SHMtfun(\bar f):\SHMtfun(Y(\varphi))\to \SHMtfun(Y'(f\circ\varphi)),\]  which then determines an isomorphism  \[\KHMtfun(f):\KHMtfun(Y,K,p)\to\KHMtfun(Y',K',p').\] This isomorphism is well-defined in that it does not depend on $\varphi$, which is to say the diagram 
\[\xymatrix@C=65pt@R=30pt{
\SHMtfun(Y(\varphi)) \ar[r]^{\SHMtfun(\bar f)} \ar[d]_{\Psit_{\varphi,\varphi'}} &  \SHMtfun(Y(f\circ\varphi)) \ar[d]^{\Psit_{f\circ \varphi,f\circ\varphi'} }\\
\SHMtfun(Y(\varphi')) \ar[r]_{\SHMtfun(\bar f)}&\SHMtfun(Y(f\circ\varphi'))   \\
} \]
commutes for all $\varphi,\varphi'$. Moreover, these isomorphisms are invariants of isotopy classes of based diffeomorphisms and respect composition  in the obvious way (see \cite[Theorem 8.5]{bs3}).

\subsection{The contact invariant in $SHM$} Suppose $(M,\Gamma)$ is a balanced sutured manifold and $\xi$ is a contact structure on $M$ such that $\partial M$ is convex and $\Gamma$ divides the characteristic foliation of $\partial M$ induced by $\xi$. We call the triple $(M,\Gamma,\xi)$  a \emph{sutured contact manifold}. 
In this subsection, we  review the construction of the contact invariant \[\invt(M,\Gamma,\xi) \in \SHMtfun(-M,-\Gamma)\] defined in \cite{bsSHM} and list some of its properties.

An important notion is that of a \emph{contact preclosure} of $(M,\Gamma)$. This is a preclosure $M'$ of $(M,\Gamma)$ together with a certain contact structure $\xi'$ on $M'$ extending $\xi$ such that $\partial M'$ is convex with dividing set $\Gamma'$ consisting of two parallel nonseparating curves on each component $\partial_{\pm}M'$. The negative region $R_-(\Gamma')$ restricted to $\partial_+M'$ is an annulus bounded by these curves, and likewise for the positive region $R_+(\Gamma')$ on $\partial_-M'$. The details of the construction of contact preclosures can be found in \cite[Subsection 3.1]{bsSHM} but are not important for this paper.

Given a contact preclosure $(M',\xi')$, one can form a closed contact 3-manifold $(Y,\bar\xi)$ by gluing $\partial_+M'$ to $\partial_-M'$ by a diffeomorphism which identifies the positive region on one component with the negative region on the other, so that the distinguished surface $R:=\partial_+M'=-\partial_-M'$ is convex with negative region an annulus. One might call such a $(Y,\bar\xi)$ a \emph{contact closure} of $(M,\Gamma,\xi)$. In \cite{bsSHM}, we gave the   slightly more involved definition below, as needed for naturality.

\begin{definition}\label{def:contactclosure}
A \emph{marked contact closure} of $(M,\Gamma,\xi)$ consists of a marked closure $\mathscr{D} = (Y,R,r,m,\eta)$ of $(M,\Gamma)$ together with a contact structure $\bar \xi$ on $Y$ such that 
\begin{enumerate}
\item $m$ restricts to a contact embedding of $(M\ssm N(\Gamma),\xi)$ into $(Y,\bar \xi)$ for some regular neighborhood $N(\Gamma)$ of $\Gamma$,
\item this restriction  of $m$ extends to a contactomorphism \[(M',\xi')\to (Y\ssm\inr(\Img(r)),\bar \xi)\] for some contact preclosure $(M',\xi')$ of $(M,\Gamma,\xi)$. 
\item $r^*(\bar\xi)$ is a $[-1,1]$-invariant contact structure on $R\times[-1,1]$.
\item the curve $r(\eta\times\{0\})$ is dual to the core of the negative annular region of $r(R\times\{0\})$. 
\end{enumerate}
\end{definition}

Suppose $(\data=(Y,R,r,m,\eta),\bar \xi)$ is a marked contact closure of $(M,\Gamma,\xi)$. As shown in \cite[Subsection 3.2]{bsSHM}, the fact that $r(R\times\{0\})$ is convex with negative region an annulus implies that \[\psi(Y,\bar\xi)\in\HMtoc(-Y,\spc_{\bar\xi};\Gamma_{-\eta})\subset \HMtoc(-Y|{-}R;\Gamma_{-\eta})=\SHMt(-\data),\]  where $\psi(Y,\bar\xi)$ is the monopole Floer contact invariant of $(Y,\bar\xi)$ defined by Kronheimer and Mrowka. This leads to the following definition from \cite{bsSHM}.

\begin{definition}
\label{def:contactinvariant} Given a  marked contact closure $(\data = (Y,R,r,m,\eta),\bar\xi)$  of $(M,\Gamma,\xi)$ of genus $g\geq g(M,\Gamma)$, we define $ \invt^g(M,\Gamma,\xi)$ to be the element of $\SHMtfun(-M,-\Gamma)$ determined by the equivalence class of  \[\invt(\data,\bar\xi):=\psi(Y,\bar\xi)\in\SHMt(-\data),\] in the sense of Remark \ref{rmk:completesubset2}. 
\end{definition}

We proved that $\invt^g(M,\Gamma,\xi)$ is well-defined for each $g$, per the following theorem.

\begin{theorem}
\label{thm:well-defined} If $(\data,\bar\xi)$ and $(\data',\bar\xi')$ are two  marked contact closures of $(M,\Gamma,\xi)$ of the same genus, then \[\Psit_{-\data,-\data'}(\invt(\data,\bar\xi)) \doteq \invt(\data',\bar\xi').\] 
\end{theorem}

Furthermore, we showed that for $g$ sufficiently large the  contact elements $\invt^g(M,\Gamma,\xi)$ are all equal, per the following theorem.

\begin{theorem}
\label{thm:well-defined2} For every $(M,\Gamma,\xi)$, there is an integer $N(M,\Gamma,\xi)$ such that \[\invt^g(M,\Gamma,\xi) = \invt^h(M,\Gamma,\xi)\] for all $g,h\geq N(M,\Gamma,\xi)$. Said differently, if  $(\data,\bar\xi)$ and $(\data',\bar\xi')$ are marked contact closures of $(M,\Gamma,\xi)$ of genus at least $N(M,\Gamma,\xi)$, then \[\Psit_{-\data,-\data'}(\invt(\data,\bar\xi)) \doteq \invt(\data',\bar\xi').\] 
\end{theorem}

This theorem motivated the following definition in \cite{bsSHM}.

\begin{definition}
\label{def:contactinvariantuniv} We define \[\invt(M,\Gamma,\xi):=\invt^g(M,\Gamma,\xi)\in\SHMtfun(-M,-\Gamma)\] for any $g\geq N(M,\Gamma,\xi)$.
\end{definition}

Below, we recall some properties of our contact invariant $\psi$ (analogous properties hold for each $\psi^g$) that were proven in \cite{bsSHM}.

\begin{proposition}
\label{prop:contactomorphism} Suppose  $f$ is a contactomorphism from $(M,\Gamma,\xi)$ to $(M',\Gamma',\xi')$. Then the induced map \[\SHMtfun(f):\SHMtfun(-M,-\Gamma)\to\SHMtfun(-M',-\Gamma')\] sends $\invt(M,\Gamma,\xi)$ to $\invt(M',\Gamma',\xi')$. In particular, the map on $\SHMtfun(-M,-\Gamma)$ induced by a contact isotopy preserves the contact invariant.
\end{proposition}

\begin{proposition}
\label{prop:ot}
If $(M,\Gamma,\xi)$ is overtwisted, then $\invt(M,\Gamma,\xi)=0$.
\end{proposition}

For the next two results, suppose $(Y,\xi)$ is a closed contact 3-manifold and $Y(1)$ is the contact 3-manifold obtained by removing a Darboux ball. The first result below is a corollary of \cite[Proposition 3.23]{bsSHM}.

\begin{proposition}
\label{prop:closeddarboux}
If $\psi(Y,\xi)\neq 0$ then $\invt(Y(1))\neq 0$.
\end{proposition}

\begin{corollary}
\label{cor:nonzerostronglyfillable}
If $(Y,\xi)$ is strongly symplectically fillable, then $\invt(Y(1))\neq 0$.
\end{corollary}

For the next result, suppose $K$ is a Legendrian knot in the interior of $(M,\Gamma,\xi)$ and  that $(M',\Gamma',\xi')$ is the result of contact $(+1)$-surgery on $K$. 

\begin{proposition}
\label{prop:shm-legendrian-surgery}
There is a morphism
\[ F_K:\SHMtfun(-M,-\Gamma) \to \SHMtfun(-M',-\Gamma') \]
which sends  $\invt(M,\Gamma,\xi)$  to $\invt(M',\Gamma',\xi')$.
\end{proposition}

\section{Invariants of Legendrian and transverse knots}
\label{sec:leginvt}

In this section, we use our invariant of sutured contact manifolds together with a prescription of Stipsicz and V{\'e}rtesi \cite{sv} to define invariants of Legendrian and transverse knots, as outlined in the introduction.

\subsection{A Legendrian invariant}
\label{ssec:leginvt}
Suppose $(K,p)$ is a based, oriented Legendrian knot in $(Y,\xi)$. Here, we define the invariant \[\kinvt(K)\in \KHMtfun(-Y,K,p)\] described in the introduction. To first approximation, $\kinvt(K)$ is the  contact invariant of the sutured contact manifold obtained by removing a standard neighborhood of $K$ and attaching a bypass  so that the resulting dividing set consists of oppositely oriented meridians. We make this precise below. 

Consider the coordinates $(\theta, (x,y))$ on $S^1\times D^2$, where $\theta\in [0,2\pi]/(0\sim 2\pi)$ is identified with  $e^{i\theta}\in S^1$ and $(x,y)$ is identified with $x+yi\in D^2$. Let $\xi_{leg}$ be the contact structure on $S^1\times D^2$ given by \[\xi_{leg} := \ker(\sin(\theta)dx + \cos(\theta)dy).\]  Every  Legendrian knot has a  neighborhood contactomorphic to $(S^1\times D^2,\xi_{leg})$; such a neighborhood  is called a \emph{standard} neighborhood. 

To define $\kinvt(K)$, we first choose a contact embedding  \[\varphi:(S^1\times D^2, \xi_{leg})\to (Y,\xi)\] such that 
$\varphi(S^1\times\{0\}) = K$ and $\varphi(\{1\}\times\{0\}) = p$. The torus $\partial(\Img(\varphi))$ is convex with respect to $\xi$. Furthermore, the dividing set  on $\partial(\Img(\varphi))$ consists of two parallel curves of slope $-1$ with respect to the coordinate system determined by   the meridian  and longitude  on $\partial (S^1\times D^2)$ given by $\mu=\{1\}\times \partial D^2$ and $\lambda =S^1\times\{1\}$, respectively. 

Let $(Y',\Gamma',\xi')$ be the sutured contact manifold obtained from  the Legendrian knot complement $(Y(\varphi),\xi|_{Y(\varphi)})$ by attaching a \emph{bypass} (roughly, a thickened neighborhood of an overtwisted disk) along the arc $\alpha\subset \partial Y(\varphi)$  shown in Figure \ref{fig:bypassvs}.  This bypass attachment does not change the underlying topology of the knot complement, but changes the dividing set (and, hence, the contact structure) in a neighborhood of the arc $\alpha$ in the manner illustrated in the figure; we refer the reader to \cite{honda2,bsSHM} for  more in-depth discussions of bypass attachments.  In particular, note that $\Gamma'$ consists of two oppositely oriented meridians, so that $(Y',\Gamma')$ and $Y(\varphi)$ are diffeomorphic. We would like to pull the contact structure $\xi'$ back to a contact structure on $Y(\varphi)$. Accordingly, let \begin{equation}\label{eqn:fmapp}f:Y(\varphi)\to (Y',\Gamma')\end{equation} be a diffeomorphism which restricts to the identity map on \[Y(\varphi)\ssm N=Y'\ssm N,\] for some standard neighborhood $N$ of $K$ with $\Img(\varphi)\subset\inr(N)$. 
Let $\xi_{K,\varphi}$ be the contact structure on $Y(\varphi)$ given by \[\xi_{K,\varphi} = (f_*)^{-1}(\xi').\] 

\begin{figure}[ht]
\labellist
\small \hair 2pt
\pinlabel $\mu$ at 85 -9
\pinlabel \rotatebox{90}{$\mu-\lambda$} at -10 82
\pinlabel $\alpha$ at 85 37
\endlabellist
\centering
\includegraphics[width=7cm]{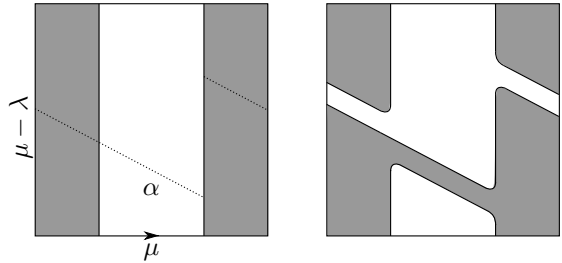}
\caption{Left, the dividing set on $\partial Y(\varphi)$ with respect to the coordinate system $(\mu,\mu-\lambda)$. Right, after attaching a bypass along  $\alpha$ to form $(Y',\xi')$.}
\label{fig:bypassvs}
\end{figure}

\begin{definition}
\label{def:legsv} Given a contact embedding $\varphi$ as above, we define $\kinvt(K)$ to be the element of $\KHMtfun(-Y,K,p)$ determined by the class
\[\kinvt(K,\varphi):=\invt(Y(\varphi),\xi_{K,\varphi})\in \SHMtfun(-Y(\varphi)).\]
\end{definition}

We prove below that $\kinvt(K)$ is well-defined:

\begin{theorem}
\label{thm:kinvt}
The element $\kinvt(K)$ is independent of the choices made in its construction.
\end{theorem}

Our proof of this theorem relies on the following technical lemma. We refer the reader to \cite{honda2} for the definition of \emph{minimally twisting}.

\begin{lemma}
\label{lem:modelunique}
Suppose $\xi$ is a tight, minimally twisting contact structure on $T^2\times [0,1]$ with  dividing set consisting of two parallel curves on each boundary component, of slope $\infty$  on $T^2\times\{0\}$ and $-1$ on $T^2\times\{1\}$. Suppose $f$ is a contactomorphism of $(T^2\times[0,1],\xi)$ which restricts to the identity on $T^2\times\{1\}$ and preserves the dividing set on $T^2\times\{0\}$. Then $f_*(\xi)$ is isotopic to $\xi$ by an isotopy which is stationary on the boundary.
\end{lemma}

\begin{proof}
The contact structure $f_*(\xi)$ is certainly tight and minimally twisting. Moreover, it has the same dividing set as $\xi$. By \cite[Proposition 4.7]{honda2}, there are exactly two tight, minimally twisting contact structures on $T^2\times[0,1]$ with this dividing set, up to isotopy stationary on the boundary, distinguished by their relative Euler classes. But  $f$ acts trivially on homology since it restricts to the identity on $T^2\times\{1\}$, which implies that the relative Euler class of $\xi$ agrees with that of $f_*(\xi)$ and, therefore, that $\xi$ and $f_*(\xi)$ are isotopic.  
\end{proof}

\begin{proof}[Proof of Theorem \ref{thm:kinvt}] We first show, for a given contact embedding $\varphi$, that the contact structure $\xi_{K,\varphi}$ on $Y(\varphi)$ is well-defined up to isotopy, so that the element $\kinvt(K,\varphi)\in\SHMtfun(-Y(\varphi))$ is well-defined by Proposition \ref{prop:contactomorphism}. 

Let $N$ be a standard neighborhood of $K$ as in the definition of $\xi_{K,\varphi}$. Note that the restriction of $\xi$ to $N\cap Y(\varphi)$ is contactomorphic to a vertically invariant contact structure on $T^2\times[0,1]$ with dividing set consisting of two slope $-1$ curves on each boundary component. It follows that the restriction of $\xi_{K,\varphi}$ to $N\cap Y(\varphi)$ is contactomorphic to a contact structure as in the hypothesis of Lemma \ref{lem:modelunique}, with relative Euler class completely pinned down (see the discussion in \cite[Section 4]{sv}). The definition of $\xi_{K,\varphi}$ depends, a priori, on the specifics of the bypass attachment and on the map $f$ in \eqref{eqn:fmapp}. However, it is clear that any two such $\xi_{K,\varphi}$ are contactomorphic by a contactomorphism restricting to the identity outside of $N$. It then follows easily from Lemma \ref{lem:modelunique} that any two such $\xi_{K,\varphi}$ are isotopic by an isotopy which is stationary on $\partial Y(\varphi)$. That $\xi_{K,\varphi}$ is independent, up to isotopy, of the neighborhood $N$ follows from the fact that for any two such $N$, there is a third which is contained in both.


We have thus shown that $\kinvt(K,\varphi)$ is well-defined. To prove that $\kinvt(K)$ is well-defined, we   must show that for any two contact embeddings $\varphi,\varphi'$,  the canonical isomorphism \[\Psit_{\varphi,\varphi'}:\SHMtfun(-Y(\varphi))\to\SHMtfun(-Y(\varphi'))\] sends $\kinvt(K,\varphi)$ to $\kinvt(K,\varphi').$ It suffices to prove this in the case that $\Img(\varphi')\subset\Img(\varphi)$. Let $N$ be a standard neighborhood of $K$ with $\Img(\varphi)\subset \inr(N)$, as above. Let $f_t$ be an ambient isotopy of $Y$, supported in $N$, of the sort used to define the map $\Psit_{\varphi,\varphi'}$. Then \[\Psit_{\varphi,\varphi'}(\kinvt(K,\varphi)):=\SHMtfun(\bar f_1)( \invt (Y(\varphi),\xi_{K,\varphi}))= \invt(Y(\varphi'), (\bar f_1)_*(\xi_{K,\varphi})),\] where the second equality is by Proposition \ref{prop:contactomorphism}. We therefore wish to show that \[\invt(Y(\varphi'), (\bar f_1)_*(\xi_{K,\varphi}))=\invt(Y(\varphi'), \xi_{K,\varphi'})=:\kinvt(K,\varphi').\] It suffices to show that $(\bar f_1)_*(\xi_{K,\varphi})$ is isotopic to $\xi_{K,\varphi'}$. But this follows from Lemma \ref{lem:modelunique} since these contact structures agree outside of $N$ and are contactomorphic on $N\cap Y(\varphi')$ to contact structures on $T^2\times[0,1]$ as in the hypothesis of the lemma, with the same relative Euler classes.

This completes the proof of Theorem \ref{thm:kinvt}.
\end{proof}

\subsection{A transverse invariant} 
\label{ssec:legprop}

Here, we prove some basic properties of the Legendrian invariant $\kinvt$ and use some of these to define the invariant of transverse knots mentioned in the introduction. 
 
Our first result is that the Legendrian invariant  behaves functorially with respect to contactomorphism in the following sense.

\begin{proposition}
\label{prop:Lfunct}
Suppose $(K,p)\subset (Y,\xi)$ and $(K',p')\subset (Y',\xi')$ are based, oriented Legendrian knots and $f:(Y,\xi)\to(Y',\xi')$ is a contactomorphism sending $(K,p)$ to $(K',p')$. Then \[\KHMtfun(f)(\kinvt(K))=\kinvt(K').\] 
\end{proposition}

\begin{proof}
Let $\varphi$ be a contact embedding of the sort used to define $\kinvt(K)$, and let $N$ be a standard neighborhood of $K$ with $\Img(\varphi)\subset \inr(N)$. The fact that $f$ is a contactomorphism implies that $f\circ \varphi$ is a contact embedding of the sort used to define $\kinvt(K')$. It suffices to prove that  \[\SHMtfun(\bar f):\SHMtfun(-Y(\varphi))\to\SHMtfun(-Y'(f\circ\varphi))\] sends $\kinvt(K,\varphi)$ to $\kinvt(K',f\circ\varphi).$ We have that \[\SHMtfun(\bar f)(\kinvt(K,\varphi)):=\SHMtfun(\bar f)(\invt(Y(\varphi),\xi_{K,\varphi})) = \invt(Y'(f\circ \varphi),(\bar f)_*(\xi_{K,\varphi})).\]  We therefore wish to show that \[\invt(Y'(f\circ \varphi),(\bar f)_*(\xi_{K,\varphi}))=\invt(Y'(f\circ \varphi),\xi_{K',f\circ\varphi})):=\kinvt(K',f\circ\varphi).\] It suffices to show that $(\bar f)_*(\xi_{K,\varphi})$ and $\xi_{K',f\circ\varphi}$ are isotopic. But this follows  from the fact that $f$ is a contactomorphism.
\end{proof}

\begin{corollary}
\label{cor:Liso}
If $(K,p)$ and $(K',p')$ are Legendrian isotopic knots  in $ (Y,\xi)$, then there exists an isomorphism \[\KHMtfun(-Y,K,p)\to\KHMtfun(-Y,K',p')\] which sends $\kinvt(K)$ to $\kinvt(K')$.\qed
\end{corollary}

The following will be important in the construction of our transverse invariant. 

\begin{proposition}
\label{prop:anyf}
Suppose $(K,p)$ and $(K',p')$ are based, oriented Legendrian knots in $(Y,\xi)$ such that there exists a contactomorphism of $(Y,\xi)$ which sends one to the other and restricts to the identity outside of  a  tubular neighborhood  $N$ of both knots. Then, for any diffeomorphism $f$ of $Y$ which restricts to the identity outside of $N$ and sends $(K,p)$ to $(K',p')$, \[\KHMtfun(f)(\kinvt(K))=\kinvt(K').\]
\end{proposition}

\begin{proof}
Let $h$ be a contactomorphism of $(Y,\xi)$ which sends $(K,p)$ to $(K',p')$ and restricts to the identity outside of $N$, whose existence is guaranteed in the hypothesis of the proposition. Let $\varphi$ be a contact embedding of the sort used to define $\kinvt(K)$, and let $N_1$ be a standard neighborhood of $K$ with $\Img(\varphi)\subset \inr(N_1)$ and $N_1\subset N$. Let $\varphi'$ be a contact embedding of the sort used to define $\kinvt(K')$ with $\Img(\varphi')\subset \Img(h\circ\varphi)\cap \Img(f\circ\varphi)$. It suffices to prove that \[\Psit_{f\circ\varphi,\varphi'}\circ\SHMtfun(\bar f):\SHMtfun(-Y(\varphi))\to\SHMtfun(-Y(\varphi'))\] sends $\kinvt(K,\varphi)$ to $\kinvt(K',\varphi').$ Let $g_t$ be an ambient isotopy of $Y$, supported in $N$, of the sort used to define $\Psit_{f\circ\varphi,\varphi'}$. Then \[\Psit_{f\circ\varphi,\varphi'}(\SHMtfun(\bar f)(\kinvt(K,\varphi))):=\SHMtfun(\bar g_1\circ \bar f)(\invt(Y(\varphi),\xi_{K,\varphi}))=\invt(Y(\varphi'),(\bar g_1\circ \bar f)_*(\xi_{K,\varphi})).\] We therefore wish to show that \[\invt(Y(\varphi'),(\bar g_1\circ \bar f)_*(\xi_{K,\varphi}))=\invt(Y(\varphi'),\xi_{K',\varphi'}):=\kinvt(K',\varphi').\] For this, it suffices to show that $(\bar g_1\circ \bar f)_*(\xi_{K,\varphi})$ is isotopic to $\xi_{K',\varphi'}$. To prove this, we first note that \[\Psit_{h\circ \varphi,\varphi'}(\SHMtfun(\bar h)(\kinvt(K,\varphi))) = \kinvt(K',\varphi')\] by Proposition \ref{prop:Lfunct} and the proof of Theorem \ref{thm:kinvt}. In fact, more is true. Let $u_t$ be an ambient isotopy of $Y$, supported in $N$, of the sort used to define $\Psit_{h\circ \varphi,\varphi'}$. It follows from the proofs of Proposition \ref{prop:Lfunct} and Theorem \ref{thm:kinvt} that $(u_1\circ \bar h)_*(\xi_{K,\varphi})$ is isotopic to $\xi_{K',\varphi'}$. So, to show that $(\bar g_1\circ \bar f)_*(\xi_{K,\varphi})$ is isotopic to $\xi_{K',\varphi'}$, it suffices to prove that \[((\bar g_1\circ \bar f)^{-1}\circ u_1\circ \bar h)_*(\xi_{K,\varphi})\] is isotopic to $\xi_{K,\varphi}$. More generally, one can show, for any diffeomorphism $F$ of $Y(\varphi)$ which restricts to the identity outside of $N$, that $F_*(\xi_{K,\varphi})$ is isotopic to $\xi_{K,\varphi}$. Indeed, up to isotopy,  $F$ restricts to the identity on $N\ssm N_1$. Then one only need compare the restrictions of $F_*(\xi_{K,\varphi})$ and $\xi_{K,\varphi}$ to the thickened torus $N_1\cap Y(\varphi)$. Lemma \ref{lem:modelunique} shows that these restrictions are isotopic by an isotopy which is stationary on $\partial N$ and thus extends to all of $Y(\varphi)$.
\end{proof}

Given a Legendrian knot $K$,  we  will denote by $K_+$ and $K_-$ its \emph{positive} and \emph{negative stabilizations}. These are certain Legendrian knots contained in a standard neighborhood of $K$ which are smoothly isotopic to $K$ in said neighborhood; see \cite{etnyre,efm} for more details. The  result below is that positive stabilization kills $\kinvt$ while negative stabilization preserves it. 

\begin{theorem}
\label{thm:natstab}
Suppose $(K,p)$ is a based, oriented Legendrian knot in $(Y,\xi)$. Let  $(K_+,p_+)$ and $(K_-,p_-)$ be stabilizations of $K$ contained within a standard neighborhood $N$ of $K$. Then $\kinvt(K_+)=0$. On the other hand,  $\kinvt(K_-)$ agrees with $\kinvt(K)$ in the sense that for any diffeomorphism $f$ of $Y$ which restricts to the identity outside of $N$ and sends $(K,p)$ to $(K_-,p_-)$,   \[\KHMtfun(f)(\kinvt(K))=\kinvt(K_-).\]
\end{theorem}

\begin{proof}
In \cite{sv}, Stipsicz and V{\'e}rtesi show that if one removes a standard neighborhood of $K_+$ from $(Y,\xi)$ and  attaches a bypass to the complement as in Figure \ref{fig:bypassvs}, then the result is overtwisted. It follows that $\kinvt(K_+)=0$.

The statement about $\kinvt(K_-)$  follows from the fact, shown in \cite{sv}, that the result of attaching a bypass to the complement of a standard neighborhood of $K$, as prescribed in Figure \ref{fig:bypassvs}, is contactomorphic to the analogous construction for $K_-$. To make this precise, let $\varphi$ be a contact embedding of the sort used to define $\kinvt(K)$ with $\Img(\varphi)\subset \inr(N)$. Let $\varphi_-$ be a contact embedding of the sort used to define $\kinvt(K_-)$ with $\Img(\varphi_-)\subset  \Img(f\circ\varphi)$. It suffices to prove that \[\Psit_{f\circ\varphi,\varphi_-}\circ\SHMtfun(\bar f):\SHMtfun(-Y(\varphi))\to\SHMtfun(-Y(\varphi_-))\] sends $\kinvt(K,\varphi)$ to $\kinvt(K_-,\varphi_-).$ Let $g_t$ be an ambient isotopy of $Y$, supported in $N$, of the sort used to define $\Psit_{f\circ\varphi,\varphi_-}$. Then, just as in the proof of Proposition \ref{prop:anyf}, it suffices to prove that  $(\bar g_1\circ \bar f)_*(\xi_{K,\varphi})$ is isotopic to $\xi_{K_-,\varphi_-}$. It follows from Stipsicz and V{\'e}rtesi's work in \cite{sv} that there exists a contactomorphism \[h:(Y(\varphi),\xi_{K,\varphi})\to(Y(\varphi_-),\xi_{K_-,\varphi_-})\] which restricts to the identity outside of $N$. So, to show that $(\bar g_1\circ \bar f)_*(\xi_{K,\varphi})$ is isotopic to $\xi_{K_-,\varphi_-}$, it suffices to prove that \[((\bar g_1\circ \bar f)^{-1}\circ  h)_*(\xi_{K,\varphi})\] is isotopic to $\xi_{K,\varphi}$. But this follows from the argument  at the end of the proof of Proposition \ref{prop:anyf}.
\end{proof}

The behavior of $\kinvt$ under negative stabilization means that we can use it to define an invariant of transverse knots via Legendrian approximation as described below. 

Consider the coordinates $(\theta, (r, \phi))$ on $S^1\times D^2$, where $\theta\in [0,2\pi]/(0\sim 2\pi)$ is identified with  $e^{i\theta}\in S^1$ and $(r,\phi)$ is identified with $re^{i\phi}\in D^2$. For $\delta >0$, let $\xi_{\delta}$ be the contact structure on $S^1\times D^2$ given by \[\xi_{\delta} := \ker(d\theta + \delta r^2d\phi)\] in these coordinates. Every  transverse knot has a  neighborhood contactomorphic to $(S^1\times D^2,\xi_{\delta})$ for some $\delta$; such a neighborhood  is called a \emph{standard} neighborhood.

Suppose $(K,p)$ is a based, oriented transverse knot in $(Y,\xi)$. We define the transverse invariant $\tinvt(K)\in\KHMtfun(-Y,K,p)$ as follows (see \cite{etnyre} for the definition of \emph{Legendrian pushoff}).

\begin{definition} Let $N$ be a standard neighborhood of $(K,p)$ and choose a Legendrian pushoff $(K',p')\subset N$ of $K$. Let $f$ be a diffeomorphism of $Y$ which restricts to the identity outside of $N$ and sends $(K',p')$ to $(K,p)$. We define \[\tinvt(K):=\KHMtfun(f)(\kinvt(K'))\in\KHMtfun(-Y,K,p).\]
\end{definition}

We prove below that $\tinvt(K)$ is well-defined:

\begin{theorem}
The element $\tinvt(K)$ is independent of the choices made in its construction.
\end{theorem}

\begin{proof}
We must show that $\tinvt(K)$ is independent of $N$, $(K',p')$, and $f$. 

First, fix $N$ and $(K',p')$. Suppose $f$ and $g$ are two diffeomorphisms of $Y$ which restrict to the identity outside of $N$ and send $(K',p')$ to $(K,p)$. Then $f^{-1}\circ g$ also restricts to the identity outside of $N$ and sends $(K',p')$ to itself. Since there exists a contactomorphism of $(Y,\xi)$ with the same property, namely the identity map, Proposition \ref{prop:anyf} implies that \[\KHMtfun(f^{-1}\circ g)(\kinvt(K')) = \kinvt(K'),\] which implies that \[\KHMtfun(f)(\kinvt(K'))=\KHMtfun(g)(\kinvt(K')).\] Thus, for  fixed $N$ and $(K',p')$, the class $\tinvt(K)$ is independent of the diffeomorphism $f$ in its definition.

Now, suppose $(K'_1,p'_1)$ and $(K'_2,p'_2)$ are  Legendrian pushoffs of $K$ in $N$. According to \cite{efm} there exists a Legendrian knot $(K''_i,p''_i)\subset N$, for each $i=1,2,$ which is the result of negatively stabilizing $(K'_i,p'_i)$ some number of times such that $(K''_1,p''_1)$ and $(K''_2,p''_2)$ are isotopic by an isotopy supported in $N$. For $i=1,2,$ let $g_i$ be a diffeomorphism of $Y$ which restricts to the identity outside of $N$ and sends $(K'_i,p'_i)$ to $(K''_i,p''_i)$. Similarly, let $h_i$ be a diffeomorphism of $Y$ which restricts to the identity outside of $N$ and sends $(K''_i,p''_i)$ to $(K,p)$. Then $h_i\circ g_i$ is a diffeomorphism of $Y$ which restricts to the identity outside of $N$ and sends $(K_i',p_i')$ to $(K,p)$. To show that $\tinvt(K)$ is independent of the Legendrian pushoff, it suffices to show that \begin{equation}\label{eqn:khmstabinvt}\KHMtfun(h_1\circ g_1)(\kinvt(K'_1))=\KHMtfun(h_2\circ g_2)(\kinvt(K'_2)).\end{equation} By Theorem \ref{thm:natstab}, we have that \[\KHMtfun(g_i)(\kinvt(K'_i))=\kinvt(K''_i)\] for $i=1,2$, so \eqref{eqn:khmstabinvt} becomes \[\KHMtfun(h_1)(\kinvt(K''_1))=\KHMtfun(h_2)(\kinvt(K''_2)),\] or, equivalently, \[\KHMtfun(h_2^{-1}\circ h_1)(\kinvt(K''_1))=\kinvt(K''_2).\] But this is true by Proposition \ref{prop:anyf} since there exists a contactomorphism of $(Y,\xi)$ which restricts to the identity outside of $N$ and sends $(K''_1,p''_1)$ to $(K''_2,p''_2)$.

We have thus shown that for a fixed $N$, the class $\tinvt(K)$ is independent of $(K',p')$ and $f$. That this class is also independent of $N$ follows from the fact that for any two such standard neighborhoods, there exists a third contained in both.
\end{proof}

The next two results are straightforward analogues of Proposition \ref{prop:Lfunct} and Corollary \ref{cor:Liso}; we omit their proofs.

\begin{proposition}
\label{prop:Tfunct}
Suppose $(K,p)\subset (Y,\xi)$ and $(K',p')\subset (Y',\xi')$ are based, oriented transverse knots and $f:(Y,\xi)\to(Y',\xi')$ is a contactomorphism sending $(K,p)$ to $(K',p')$. Then \[\KHMtfun(f)(\tinvt(K))=\tinvt(K').\] 
\end{proposition}

\begin{corollary}
\label{cor:Tiso}
If $(K,p)$ and $(K',p')$ are transversely isotopic knots  in $ (Y,\xi)$, then there exists an isomorphism \[\KHMtfun(-Y,K,p)\to\KHMtfun(-Y,K',p')\] which sends $\tinvt(K)$ to $\tinvt(K')$.
\end{corollary}


\subsection{Additional properties of $\kinvt$} Here, we describe some additional properties satisfied by $\kinvt$ which are analogous to those satisfied by the LOSS invariant \cite{lossz,sahamie,ost} and the second author's invariant $\ell$ from \cite{sivek}. We will omit basepoints from our notation for convenience and because they are not so relevant to the results. The first result below follows almost exactly as in the proof of 
Proposition \ref{prop:shm-legendrian-surgery} (given in  \cite{bsSHM}); we will therefore omit it.

\begin{proposition}
\label{prop:legsurgknot}
Let $K,S \subset (Y,\xi)$ be disjoint Legendrian knots, and let $(Y',\xi')$ be the contact manifold obtained by a contact $(+1)$-surgery along $S$.  If $K$ has image $K'$ in $Y'$, then there is a map
\[ \KHMtfun(-Y,K) \to \KHMtfun(-Y',K') \]
which sends $\kinvt(K)$ to $\kinvt(K')$.\qed
\end{proposition}

The next proposition is helpful in computing the invariant $\kinvt$ in certain cases.

\begin{proposition}
\label{prop:legunknotmap}
Suppose $U\subset (Y,\xi)$ is a  Legendrian unknot with $tb(U)=-1$ contained inside a Darboux ball in $(Y,\xi)$. Then there is an isomorphism \[\SHMtfun(-Y(1))\to\KHMtfun(-Y,U)\] which sends $\invt(Y(1))$ to $\kinvt(U)$.
\end{proposition}

\begin{proof}
Let $B^3\subset Y$ be a Darboux ball containing $U$ and let $Y(1)$ refer to the sutured contact manifold obtained by removing a smaller Darboux ball $(B^3)'\subset B^3$ from $Y$. Note that one can attach a contact $1$-handle to $Y(1)$ so as to obtain a sutured manifold, often denoted by $Y(U)$, which is  the complement of a regular neighborhood $N\subset B^3$ of $U$ with two meridional sutures. One can identify $Y(U)$ with $Y(\varphi)$ for some embedding $\varphi$ as in Subsection \ref{ssec:khm}. As shown in \cite[Subsection 4.2]{bsSHM}, this contact $1$-handle attachment  gives rise to an isomorphism  \[\SHMtfun(-Y(1))\to\SHMtfun(-Y(U))\cong\KHMtfun(-Y,U).\]

Let $\xi_U$ be the contact structure on $Y(U)$ induced from that on $Y(1)$. It follows from \cite[Corollary 4.14]{bsSHM}  that the map above sends $\invt(Y(1))$ to $\invt(Y(U),\xi_U)$. To complete the proof of Proposition \ref{prop:legunknotmap}, it suffices to check that $(Y(U),\xi_U)$ is contactomorphic to the contact structure on $Y(U)$ which defines $\kinvt(U)$; namely, the contact structure obtained by removing a standard neighborhood of $U$ (which we can assume is contained in $B^3$) and then attaching a bypass as described in Subsection \ref{ssec:leginvt}. Since both this construction and the $1$-handle attachment above are performed in a Darboux ball, it suffices to check this in the case that $(Y,\xi) = (S^3,\xi_{std})$. In this case, $Y(U)$ is a solid torus with two longitudinal sutures. As there is a unique isotopy class of tight contact structures on such a torus, it suffices to check that both $\xi_U$ and the contact structure defining $\kinvt(U)$ are tight. We know that $\xi_U$ is tight since $\invt(Y(U),\xi_U)$ is identified with $\invt(Y(1))$ by the isomorphism above, and $\invt(Y(1))\neq 0$ in this case by, for example, Corollary \ref{cor:nonzerostronglyfillable}. The fact that the contact structure defining $\kinvt(U)$ is tight follows from the fact that the LOSS invariant of the $tb=-1$ unknot in $(S^3,\xi_{std})$ is nonzero and the fact that this LOSS invariant is the Heegaard Floer contact invariant of the contact structure of interest \cite{sv}. (It is possible to give a slightly longer, but self-contained---i.e., independent of results in Heegaard Floer homology---proof.)
\end{proof}

Combined with Proposition \ref{prop:closeddarboux}, this implies the following.

\begin{corollary}
If $\psi(Y,\xi)\neq 0$, then $\kinvt(U)\neq 0$.\qed
\end{corollary}

Propositions \ref{prop:legsurgknot} and \ref{prop:legunknotmap} admit the following additional corollaries. The proof of each  is identical to that of the corresponding statement in \cite[Section 5]{sivek}, with $\kinvt$ in place of $\ell$. The first result below is an analogue of a theorem of Sahamie \cite[Theorem 6.1]{sahamie}.

\begin{corollary}
\label{cor:sahamiecor}
Suppose $K$ is a Legendrian knot in $(Y,\xi)$ and let $(Y',\xi')$ be the result of contact $(+1)$-surgery on $K$. Then there exists a map
\[\KHMtfun(-Y,K)\to \SHMtfun(-Y'(1))\] which sends $\kinvt(K)$ to $\invt(Y'(1))$. Thus, if $\psi(Y',\xi')\neq 0$, then $\kinvt(K)\neq 0$. \qed
\end{corollary}

The following is valid only over $\RR/2\RR:=\RR\otimes_\Z \Z/2\Z$ (its proof involves the surgery exact triangle in monopole Floer homology which has only been established in characteristic two).

\begin{corollary}
\label{cor:kinvtslice}
Suppose $K$ is a Legendrian knot in $(S^3,\xi_{std})$ with slice genus $g_s(K)>0$ and $tb(K)=2g_s-1$. Then $\kinvt(K)\neq 0$.\qed
\end{corollary}

Finally, note that Proposition \ref{prop:ot} immediately implies the following.

\begin{corollary}
\label{cor:loose}
Suppose $K$ is a Legendrian knot in $(Y,\xi)$ and the complement of  $K$ is overtwisted. Then $\kinvt(K)= 0$. \qed
\end{corollary}

\section{A map induced by Lagrangian concordances}
\label{sec:lagcobordism}

The following relation on the set of Legendrian knots in a closed contact 3-manifold $(Y,\xi)$ was introduced by Chantraine \cite{chantraine}.

\begin{definition}
\label{def:lagconcordance}
Let $K_-$ and $K_+$ be Legendrian knots in $(Y,\xi)$ parametrized by maps $\gamma_\pm: S^1 \to Y$.  We say that $K_-$ is \emph{Lagrangian concordant} to $K_+$ if there is a  Lagrangian cylinder
\[ L: S^1 \times \R \hookrightarrow Y\times \R \]
in the symplectization of $Y$ and a constant $T>0$ such that $L(s,t) = \gamma_-(s)$ for all $t\leq -T$ and $L(s,t)  = \gamma_+(s)$ for all $t\geq T$. We will $L$ to refer both to this map and to its image.
\end{definition}

Chantraine showed that $tb$ and $r$ are Lagrangian concordance invariants, and the second author showed in \cite{sivek} that the Legendrian invariant $\ell$ is well-behaved under Lagrangian concordance as well.  In this section we will prove similar results about the effect of Lagrangian concordance on the invariant $\kinvt$. Our main result is the following.

\begin{theorem}
\label{thm:lagconcordance}
Suppose $K_-$ and $K_+$ are Legendrian knots in $(Y,\xi)$ and  $K_-$ is Lagrangian concordant to $K_+$. Then there is a map
\[\KHMtfun(-Y,K_+) \to \KHMtfun(-Y,K_-) \]
which sends $\kinvt(K_+)$ to $\kinvt(K_-)$.
\end{theorem}

The idea behind the proof  is to replace a model neighborhood of the concordance $L\subset Y\times I$ with a certain symplectization so as to obtain an exact symplectic cobordism from a contact closure of the  complement of $K_-$ to a contact closure of the complement of $K_+$, where these complements are equipped with the contact structures used to define the invariants $\kinvt(K_\pm)$. We then use the fact that the monopole Floer contact invariant for closed contact manifolds is functorial with respect to the maps induced by exact symplectic cobordisms (see \cite{ht} and the discussion in \cite[Subsection 2.3]{bsSHM}).

\begin{proof}[Proof of Theorem \ref{thm:lagconcordance}]
Let $\alpha$ be a contact form for $(Y,\xi)$. By Definition \ref{def:lagconcordance}, there exists a Lagrangian cylinder $L\subset (Y\times I, d(e^t\alpha))$ with $L\cap \{\pm T\} = K_\pm\times\{\pm T\}$, where $I = [-T-\epsilon, T+\epsilon]$ for some $\epsilon,T>0$. Observe that $L$ is \emph{exact}, meaning that the form $e^t\alpha|_L$ is exact. It suffices to show that $e^t\alpha$ vanishes on $H_1(L;\R)$. For this,  note that $H_1(L;\R)$ is generated by the class of $K_-\times\{-T\}$ and  \[ \int_{K_- \times \{-T\}} e^t\alpha = e^{-T} \int_{K_-} \alpha = 0 \] since $K_-$ is Legendrian. The Lagrangian cylinder $L_-:=K_-\times I\subset (Y\times I, d(e^t\alpha))$ is similarly exact. The Weinstein Tubular Neighborhood Theorem therefore gives an \emph{exact} symplectomorphism \[\varphi:N(L_-)\to N(L),\] from a neighborhood of one Lagrangian to a neighborhood of the other. Exactness in this context means that if $d\lambda_-$ and $d\lambda$ are symplectic forms on the neighborhoods $N(L_-)$ and $N(L)$, then $\varphi^*\lambda-\lambda_-=df$ for some function $f:N(L_-)\to \R$. We are free to choose the identification $\varphi|_{L_-}:L_-\xrightarrow{\sim} L$ to be the identity on $K_-\times[-T-\epsilon,-T]$, so that $df\equiv0$ on $N(L_-)\cap (Y\times[-T-\epsilon,-T]).$ Since $f$ is only determined up to a constant, we may then require that $f\equiv 0$ on $N(L_-) \cap (Y\times[-T-\epsilon,-T])$.

By shrinking $N(L_-)$ and $N(L)$ if necessary, we can assume that $N(L_-)$ is of the form $N(K_-)\times I$, where $N(K_-)$ is a standard neighborhood of $K_-$ in $(Y,\xi)$. Let us identify $N(K_-)$ with  $(S^1\times D^2_2,\xi_{leg})$, where $D^2_a\subset \C$ refers to the disk centered at the origin of radius $a$, so that $\varphi$ can be thought of as a map
\[\varphi:(S^1\times D^2_2)\times I \to N(L).\] Note that $\varphi$ sends $(S^1\times D^2_2)\times\{\pm (T+\epsilon)\}$ to  a standard neighborhood of $K_\pm$. 

Let $Y_-$ be the sutured contact manifold obtained from the sutured Legendrian knot complement $Y\ssm (S^1\times D^2_1)$ by attaching a bypass as prescribed in Subsection \ref{ssec:leginvt} (for the construction of $\kinvt(K_-)$). We can assume that the contact structure on $Y_-$ agrees with $\xi$ outside of $S^1\times D^2_{1.5}$. Let $\data_- = ((\bar Y_-,R_-, r_-, m_-, \eta_-), \bar\xi_-)$  be a marked contact closure of $Y_-$, where $\bar Y_-$ is built from a contact preclosure of $Y_-$ by attaching a $[-1,1]$-invariant contact structure on  $R_-\times [-1,1]$, so that $r_-$ and $m_-$ are inclusion maps. 
Let \[Z = \bar Y_-\ssm (Y_-\ssm (S^1\times D^2_2)),\] and let $A$ refer to   the regions of $Y$ and $Z$ given by $S^1\times(D^2_2\ssm D^2_{1.5})$. Let $(X,\omega)$ be the symplectic manifold formed by gluing the symplectization $Z\times I$ to  the complement $(Y\times I) \ssm \varphi((S^1\times D^2_{1.5})\times I)$  by the map \[\varphi:A\times I\to \varphi(A\times I),\] as depicted in Figure \ref{fig:concordancegluing}.

\begin{figure}[ht]
\labellist
\small \hair 2pt
\pinlabel $\varphi$ at 244 268
\tiny
\pinlabel $1$ at 103 417
\pinlabel $1.5$ at 103 440
\pinlabel $2$ at 103 463
\pinlabel $1.5$ at 103 187
\pinlabel $2$ at 103 209
\endlabellist
\centering
\includegraphics[width=7.5cm]{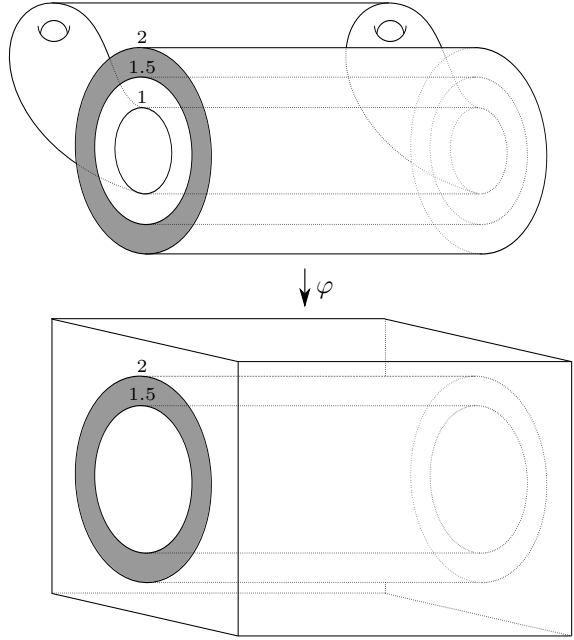}
\caption{Top, a schematic of $Z$, with the region $A$ shown in gray. Bottom, a schematic of $(Y\times I) \ssm \varphi((S^1\times D^2_{1.5})\times I)$, with $\varphi(A)$ shown in gray. }
\label{fig:concordancegluing}
\end{figure}

Let  $Y_\pm^{\varphi}$ be the sutured contact manifold obtained from the sutured Legendrian knot complement $(Y\times \{\pm(T+\epsilon)\}) \ssm \varphi((S^1\times D^2_1)\times\{\pm (T+\epsilon)\})$ by attaching by a bypass as prescribed in Subsection \ref{ssec:leginvt}. By construction,  $(X,\omega)$ is a symplectic cobordism from $(\bar Y_-^\varphi,\bar\xi_-^\varphi)$ to $(\bar Y_+^\varphi,\bar\xi_+^\varphi)$, where $(\bar Y_\pm^\varphi,\bar\xi_\pm^\varphi)$ is the   contact manifold underlying a marked contact closure \[\data_\pm^\varphi = ((\bar Y_\pm^\varphi,R_\pm^\varphi, r_\pm^\varphi, m_\pm^\varphi, \eta_\pm^\varphi), \bar\xi_\pm^\varphi)\] of $Y_\pm^{\varphi}$ (the additional data  is naturally inherited from $\data_-$). In particular, the Legendrian invariant $\kinvt(K_\pm)\in\KHMtfun(-Y,K_\pm)$ is  represented by the contact class \[\invt(\bar Y_\pm^\varphi,\bar\xi_\pm^{\varphi})\in \SHMt(-\data_{\pm}^\varphi)=\HMtoc(-\bar Y_\pm^{\varphi}|{-}R_\pm^\varphi;\Gamma_{{-}\eta_\pm^\varphi}).\] To complete the proof of Theorem \ref{thm:lagconcordance}, it therefore suffices to show that the induced map \begin{equation}\label{eqn:mapX}\HMtoc(X|{-}R_-^\varphi;\Gamma_{{-}\nu}):\HMtoc(-\bar Y_+^{\varphi}|{-}R_+^\varphi;\Gamma_{{-}\eta_+^\varphi})\to\HMtoc(-\bar Y_-^{\varphi}|{-}R_-^\varphi;\Gamma_{{-}\eta_-^\varphi})\end{equation} sends $\invt(\bar Y_+^\varphi,\bar\xi_+^{\varphi})$ to $\invt(\bar Y_-^\varphi,\bar\xi_-^{\varphi})$, up to multiplication by a unit in $\RR$, where $\nu = \eta^{\varphi}_-\times I\subset Z\times I\subset X$ is a cylindrical cobordism from $\eta^{\varphi}_-$ to $\eta^\varphi_+$. This would follow from the work in \cite{ht} (see also \cite[Theorem 2.22 \& Remark 2.24]{bsSHM})  if we knew that $(X,\omega)$ were an \emph{exact} symplectic cobordism from $(\bar Y_-^\varphi,\bar\xi_-^\varphi)$ to $(\bar Y_+^\varphi,\bar\xi_+^\varphi)$. This is true after a slight modification (replacing $X$ with $\tilde X$ below).

We first show that the 2-form $\omega$ is exact. Suppose $\lambda_-$ and $\lambda$ are global primitives for the symplectic forms on the pieces $Z\times I$ and $(Y\times I) \ssm \varphi((S^1\times D^2_{1.5})\times I)$ coming from the Liouville vector field $\partial_t$. Since $\varphi$ is an exact symplectomorphism, we have that $\varphi^*\lambda-\lambda_- = df$ and $d\lambda = \omega$ on the intersection of these pieces, which we will identify as $A\times I$. Let $\rho:A\to[0,1]$ be a smooth cutoff function which is equal to 1 on a neighborhood of $S^1\times D^2_2$ and 0 on a neighborhood of $S^1\times D^2_{1.5}$. Define the 1-form 
\[ \lambda_X = \begin{cases}
\lambda_- & \mathrm{on\ } (Z\smallsetminus A) \times I \\
\lambda_- + d(\rho f) & \mathrm{on\ } A \times I \\
\lambda & \mathrm{on\ } (Y\times I) \ssm \varphi((S^1\times D^2_{2})\times I). \\
\end{cases} \]
Then $\lambda_X$ is a globally defined primitive for $\omega$, and its restriction to the boundary component $\bar Y_-^\varphi$ is a contact form for $\bar \xi_-^\varphi$, since    $f\equiv 0$ near this boundary component. On the other hand,
$ \omega|_{Y_+}  = d\lambda_+ $
for some contact form $\lambda_+$ for $\bar \xi_+^\varphi$, but we do not know that $\lambda_X|_{Y_+}$ is itself a contact form because of the $d(\rho f)$ term. To remedy this, we  apply a result of Eliashberg \cite[Proposition 3.1]{yasha7}. His result says that we can glue a symplectic $(\bar Y_+^{\varphi} \times [1,C], \Omega)$ to $\bar Y_+^\varphi \subset \partial X$ so that $\Omega$ is an exact 2-form whose primitive agrees with the primitive $\lambda_X$ of $\omega$ near $\bar Y_+^\varphi \times \{1\}$ and $\Omega$ is the symplectization of $\bar Y_+^{\varphi}$ near $\bar Y_+^\varphi\times \{C\}$.  The result is a symplectic form $\tilde{\omega} = d\tilde{\lambda}$ on the cobordism
\[ \tilde{X} = X \cup_{\bar Y_+^\varphi} (\bar Y_+^\varphi \times [1,C]) \] from $\bar Y_-^\varphi$ to $\bar Y_+^\varphi$
so that $\tilde{\lambda}$ restricts  to contact forms for $(\bar Y_-^\varphi,\bar\xi_-^\varphi)$ and $(\bar Y_+^\varphi,\bar\xi_+^\varphi)$. Replacing $X$ with $\tilde X$ in \eqref{eqn:mapX}, we have the desired result.
\end{proof}

\begin{remark}
A \emph{decorated}  concordance $(L,\gamma)\subset Y\times I$ from $(K_-,p_-)$ to $(K_+,p_+)$ is a concordance $L$ from $K_-$ to $K_+$ together with an embedded arc $\gamma\subset L$ from $p_-$ to $p_+$. Based on Juh{\'a}sz's work  \cite{juhasz3}, we expect that such an $(L,\gamma)$ should induce a well-defined map (i.e., one that is independent of the auxiliary choices in its construction) \[F_{(L,\gamma)}:\KHMtfun(-Y,K_+,p_+)\to\KHMtfun(-Y,K_-,p_-),\] defined in a manner similar to the construction of the map  in Theorem \ref{thm:lagconcordance}. 
More generally, we expect that decorated \emph{cobordisms} between based knots should induced well-defined maps on $\KHMtfun$, and that there should be an analogue of Theorem \ref{thm:lagconcordance} for exact Lagrangian cobordisms of arbitrary genus. It is   worth noting that, even without such an analogue, we can prove the following.
 \end{remark}

\begin{lemma}
Suppose $U$ is the Legendrian unknot in $(S^3,\xi_{std})$ with $tb=-1$.
If $U$ is exact Lagrangian cobordant  to a Legendrian knot $K$ with $g_s(K)>0$, then $\kinvt(K) \neq 0$ over $\RR/2\RR$.
\end{lemma}

\begin{proof}
The existence of such a cobordism implies, by \cite[Theorem 1.3]{chantraine}, that $tb(K) = 2g_s(K)-1$.  Corollary \ref{cor:kinvtslice} then implies that $\kinvt(K) \neq 0$ over $\RR/2\RR$.
\end{proof}

\section{Examples and nonreversible Lagrangian concordances}
\label{sec:examples}
We end with some examples. The first two  illustrate how Theorem \ref{thm:lagconcordance} can be used to deduce the nonvanishing of  $\kinvt$.  All provide new examples of nonreversible Lagrangian concordances.

\begin{example} For the first example (suggested by Lenny Ng), consider the Legendrian knots represented by the grid diagrams in Figure \ref{fig:m8_20} (which can be converted to front diagrams by smoothing all northeast and southwest corners and rotating $45^\circ$ counterclockwise). This figure describes a Lagrangian concordance from a negative stabilization $U_-$ of the $tb=-1$ Legendrian unknot  to a Legendrian representative $K$ of the knot $m(8_{20})$ with $(tb(K),r(K))=(-2,-1)$. Proposition \ref{prop:legunknotmap} tells us that $\kinvt(U)\neq 0$. Theorem \ref{thm:natstab} then implies that $\kinvt(U_-)\neq 0$.  We may  therefore conclude from Theorem \ref{thm:lagconcordance} that $\kinvt(K)\neq 0$ as well. 

\label{ex:m8_20}
\begin{figure}[ht]
\centering
\includegraphics[width=9.4cm]{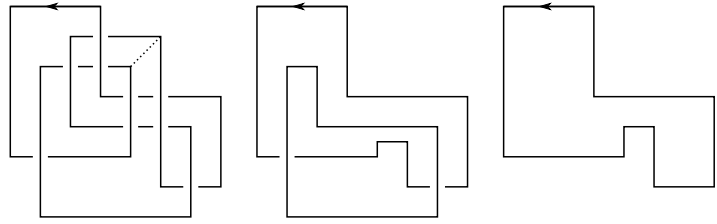}
\caption{Inserting a Lagrangian saddle \cite{ehk} along the dotted line and then capping off a Legendrian unknot to build a Lagrangian concordance from  $U_-$ to a Legendrian $m(8_{20})$, shown here in reverse.}
\label{fig:m8_20}
\end{figure}

\end{example}

\begin{example}
For the next example, consider Figure \ref{fig:m11n_71}, which describes a Lagrangian concordance from a negative stabilization $C_-$ of a Legendrian representative $C$ of $m(5_2)$ with $tb(C) = \overline{tb}(m(5_2)) = 1$ to a Legendrian representative $K$ of $m(11n_{71})$ with $tb(K) = \overline{tb}(m(11n_{71})) = 0$ and $r(K)=-1$. The knot $m(5_{2})$ has smooth slice genus $g_s=1$, and $tb(C) = 2g_s(C)-1$, so $\kinvt(C)\neq 0$ over $\RR/2\RR$ by Corollary \ref{cor:kinvtslice}. Theorem \ref{thm:natstab} then implies that $\kinvt(C_-)\neq 0$ over $\RR/2\RR$. We may therefore conclude that $\kinvt(K)\neq 0$ over $\RR/2\RR$, by Theorem \ref{thm:lagconcordance}. 

\label{ex:m11n_71}
\begin{figure}[ht]
\centering
\includegraphics[width=9.9cm]{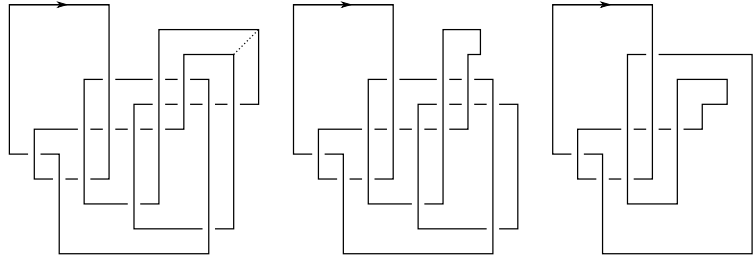}
\caption{A Lagrangian concordance from a negatively stabilized Legendrian $m(5_2)$ to a Legendrian $m(11n_{71})$, shown in reverse.}
\label{fig:m11n_71}
\end{figure}
\end{example}

\begin{remark} It is worth noting that the strategies used in the two examples above to deduce  nonvanishing results for $\kinvt$ cannot be used to deduce the analogous nonvanishing of Legendrian contact homology ($LCH$) for these Legendrian representatives of $m(8_{20})$ and $m(11n_{71})$: although $LCH$ behaves naturally with respect to Lagrangian concordance \cite{ehk}, it vanishes for stabilized knots. 

On the other hand, $LCH$ can be used to show that the Lagrangian concordances described in these examples are nonreversible: if there were  a  concordance from any of these $K$ to the respective stabilized knot $S$, then by \cite{ehk} there would be a morphism $f: \mathcal{A}(S) \to \mathcal{A}(K)$, where $\mathcal{A}$ refers to the Legendrian contact homology DGA.  Since $S$ is stabilized, there exists $x\in \mathcal{A}(S)$ such that $\partial x = 1$, and so we would have
\[ \partial(f(x)) = f(\partial x) = f(1) = 1; \]
i.e.\ $\mathcal{A}(K)$ would be trivial.  But each $K$ in the examples above achieves the Kauffman bound on $tb$, so $\mathcal{A}(K)$ admits an ungraded augmentation \cite{rutherford} and is thus necessarily nontrivial.
\end{remark}

The first examples of nonreversible Lagrangian concordances were found by Chantraine in \cite{chantraine3}. Our examples above provide new instances of this phenomenon. Given a nonreversible Lagrangian concordance, one can find infinitely many distinct such concordances by connect summing. 
 Below, we provide the first infinite family of nonreversible Lagrangian concordances between \emph{prime} knots. These examples also indicate that results in $LCH$ may sometimes provide more information than our results for $\kinvt$.

\begin{example}
Figure \ref{fig:pretzels} shows the first four members $K_1,K_2,K_3,K_4$ of an infinite family of Legendrian knots, where $K_n$ is a Legendrian representative the  pretzel knot $P(n,3,-3)$. 
It is straightforward to compute that $tb(K_n) = -(n+4)$, and that $r(K_n)$ is 0 for odd $n$ and $\pm 1$, depending on a choice of orientation, for even $n$. 

\label{ex:pretzels}
\begin{figure}[ht]
\centering
\includegraphics[width=12cm]{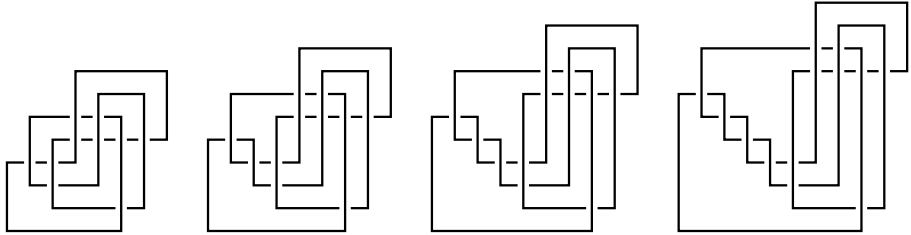}
\caption{A family  of Legendrian pretzel knots.}
\label{fig:pretzels}
\end{figure}

For each $n$, there is a Lagrangian concordance from a Legendrian unknot $U_n$  to $K_n$, as illustrated  in Figure \ref{fig:p4-surgery} for $n=4$. Note that $\kinvt(U_n) = 0$ by Theorem \ref{thm:natstab} since $U_n$ is a positive stabilization, so we cannot use Theorem \ref{thm:lagconcordance} to determine whether $\kinvt(K_n)$ vanishes. On the other hand, each $K_n$ admits an ungraded normal ruling,  shown in Figure \ref{fig:p4-ruling} for $n=4$, which implies that the Legendrian contact homology $\mathcal{A}(K_n)$ admits an ungraded augmentation \cite{fuchs}. Thus, $\mathcal{A}(K_n)$ is nonvanishing. Since $\mathcal {A}(U_n)$ is trivial but $\mathcal{A}(K_n)$ is not, it follows that the Lagrangian concordance from $U_n$ to $K_n$ is  nonreversible for each $n$.

\begin{figure}[ht]
\centering
\includegraphics[width=12.5cm]{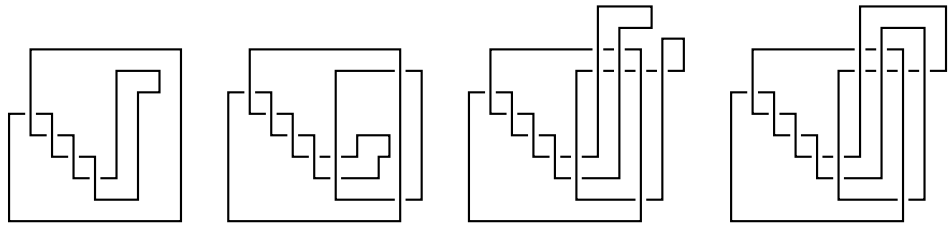}
\caption{A Lagrangian concordance from $U_4$ to $K_4$, built by inserting a Legendrian unknot filled with a Lagrangian cap, performing a Legendrian isotopy, and then using a Lagrangian saddle to surger two cusps together.}
\label{fig:p4-surgery}
\end{figure}

\begin{figure}[ht]
\centering
\includegraphics[width=6.5cm]{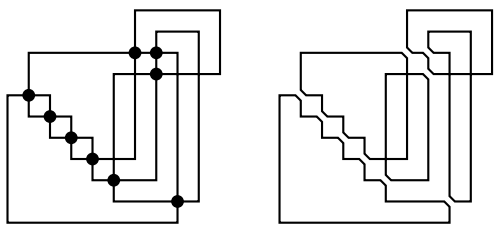}
\caption{An ungraded normal ruling of $K_4$.}
\label{fig:p4-ruling}
\end{figure}

\begin{remark} We expect that $\kinvt(K_n)=0$ for all $n$. Indeed, each $K_n$ has thin Khovanov homology \cite{starkston}. Hence, by a well-known conjecture (cf.\ \cite[Conjecture 1.3]{baldlev} and \cite{ras2}), it should have thin knot Floer homology as well (and does for $n$ odd, as computed in \cite{osz11}).  Ng, Ozsv{\'a}th and Thurston showed in \cite{not}  that the GRID invariant \cite{oszt} of an $HFK$-thin Legendrian knot with $tb-r < 2\tau-1$ vanishes. This inequality holds for $K_n$ since  $K_n$ is slice. In future work we intend to identify $\kinvt(K_n)$ with the GRID invariant, so it should follow that $\kinvt(K_n)=0$ as well.
\end{remark}

\end{example}

\bibliographystyle{hplain}
\bibliography{References}

\begin{thebibliography}{10}

\bibitem{baldlev}
John~A. Baldwin and Adam~Simon Levine.
\newblock A combinatorial spanning tree model for knot {F}loer homology.
\newblock {\em Adv. Math.}, 231(3-4):1886--1939, 2012.

\bibitem{bsSHM}
John~A. Baldwin and Steven Sivek.
\newblock A contact invariant in sutured monopole homology.

\bibitem{bs3}
John~A. Baldwin and Steven Sivek.
\newblock Naturality in sutured monopole and instanton homology.
\newblock 2013, math.GT/1305.4156.

\bibitem{bvv}
John~A. Baldwin, David~Shea Vela-Vick, and Vera V{\'e}rtesi.
\newblock On the equivalence of {L}egendrian and transverse invariants in knot
  {F}loer homology.
\newblock {\em Geom. Topol.}, 17:925--974, 2013.

\bibitem{chantraine}
Baptiste Chantraine.
\newblock Lagrangian concordance of {L}egendrian knots.
\newblock {\em Algebr. Geom. Topol.}, 10(1):63--85, 2010.

\bibitem{chantraine3}
Baptiste Chantraine.
\newblock Lagrangian concordance is not a symmetric relation.
\newblock 2013, math.SG/1301.3767.

\bibitem{cgh5}
V.~Colin, K.~Honda, and P.~Ghiggini.
\newblock {The equivalence of Heegaard Floer homology and embedded contact
  homology III: from hat to plus}.
\newblock 2012, math.SG/1208.1526.

\bibitem{cgh3}
V.~Colin, K.~Honda, and P.~Ghiggini.
\newblock {The equivalence of Heegaard Floer homology and embedded contact
  homology via open book decompositions I}.
\newblock 2012, math.SG/1208.1074.

\bibitem{cgh4}
V.~Colin, K.~Honda, and P.~Ghiggini.
\newblock {The equivalence of Heegaard Floer homology and embedded contact
  homology via open book decompositions II}.
\newblock 2012, math.SG/1208.1077.

\bibitem{eens}
Tobias Ekholm, John Etnyre, Lenhard Ng, and Michael Sullivan.
\newblock Filtrations on the knot contact homology of transverse knots.
\newblock {\em Math. Ann.}, 355(4):1561--1591, 2013.

\bibitem{ehk}
Tobias Ekholm, Ko~Honda, and Tam\'{a}s K\'{a}lm\'{a}n.
\newblock Legendrian knots and exact {L}agrangian cobordisms.
\newblock 2012, arXiv:1212.1519.

\bibitem{yasha7}
Yakov Eliashberg.
\newblock On symplectic manifolds with some contact properties.
\newblock {\em J. Differential Geom.}, 33(1):233--238, 1991.

\bibitem{yasha8}
Yakov Eliashberg.
\newblock Invariants in contact topology.
\newblock In {\em Proceedings of the {I}nternational {C}ongress of
  {M}athematicians, {V}ol. {II} ({B}erlin, 1998)}, number Extra Vol. II, pages
  327--338, 1998.

\bibitem{efm}
J.~Epstein, D.~Fuchs, and M.~Meyer.
\newblock {Chekanov-Eliashberg invariants and transverse approximations of
  Legendrian knots}.
\newblock {\em Pac. J. Math.}, 201(1):89--106, 2001.

\bibitem{etnyre}
John~B. Etnyre.
\newblock Legendrian and transversal knots.
\newblock In {\em Handbook of knot theory}, pages 105--185. Elsevier B. V.,
  Amsterdam, 2005.

\bibitem{fuchs}
Dmitry Fuchs.
\newblock Chekanov-{E}liashberg invariant of {L}egendrian knots: existence of
  augmentations.
\newblock {\em J. Geom. Phys.}, 47(1):43--65, 2003.

\bibitem{honda2}
K.~Honda.
\newblock On the classification of tight contact structures, {I}.
\newblock {\em Geom. Topol.}, 4:309--368, 2000.

\bibitem{hkm4}
Ko~Honda, William~H. Kazez, and Gordana Mati{\'c}.
\newblock The contact invariant in sutured {F}loer homology.
\newblock {\em Invent. Math.}, 176(3):637--676, 2009.

\bibitem{ht}
Michael Hutchings and Clifford~Henry Taubes.
\newblock Proof of the {A}rnold chord conjecture in three dimensions, {II}.
\newblock {\em Geom. Topol.}, 17(5):2601--2688, 2013.

\bibitem{juhaszthurston}
A.~Juh{\'a}sz and D.~P. Thurston.
\newblock {Naturality and mapping class groups in Heegaard Floer homology}.
\newblock 2012, math.GT/1210.4996.

\bibitem{juhasz3}
Andr{\'a}s Juh{\'a}sz.
\newblock Cobordisms of sutured manifolds.
\newblock 2009, math.GT/0910.4382.

\bibitem{km}
P.~Kronheimer and T.~Mrowka.
\newblock Monopoles and contact structures.
\newblock {\em Inv. Math.}, 130:209--255, 1997.

\bibitem{kmosz}
P.~Kronheimer, T.~Mrowka, P.~Ozsv{\'a}th, and Z.~Szab{\'o}.
\newblock Monopoles and lens space surgeries.
\newblock {\em Ann. Math.}, 165(2):457--546, 2007.

\bibitem{km4}
Peter Kronheimer and Tomasz Mrowka.
\newblock Knots, sutures, and excision.
\newblock {\em J. Differential Geom.}, 84(2):301--364, 2010.

\bibitem{klt1}
C.~Kutluhan, Y.-J. Lee, and C.~H. Taubes.
\newblock {HF = HM I: Heegaard Floer homology and Seiberg-Witten Floer
  homology}.
\newblock 2010, math.SG/1007.1979.

\bibitem{klt2}
C.~Kutluhan, Y.-J. Lee, and C.~H. Taubes.
\newblock {HF = HM II: Reeb orbits and holomorphic curves for the
  ech/Heegaard-Floer correspondence}.
\newblock 2010, math.SG/1008.1595.

\bibitem{klt3}
C.~Kutluhan, Y.-J. Lee, and C.~H. Taubes.
\newblock {HF = HM III: Holomorphic curves and the differential for the
  ech/Heegaard-Floer correspondence}.
\newblock 2010, math.SG/1010.3456.

\bibitem{klt4}
C.~Kutluhan, Y.-J. Lee, and C.~H. Taubes.
\newblock {HF = HM IV: The Seiberg-Witten Floer homology and ech
  correspondence}.
\newblock 2011, math.GT/1107.2297.

\bibitem{klt5}
C.~Kutluhan, Y.-J. Lee, and C.~H. Taubes.
\newblock {HF = HM V: Seiberg-Witten Floer homology and handle additions}.
\newblock 2012, math.GT/1204.0115.

\bibitem{lekili2}
Yank{\i} Lekili.
\newblock Heegaard-{F}loer homology of broken fibrations over the circle.
\newblock {\em Adv. Math.}, 244:268--302, 2013.

\bibitem{lossz}
Paolo Lisca, Peter Ozsv{\'a}th, Andr{\'a}s~I. Stipsicz, and Zolt{\'a}n
  Szab{\'o}.
\newblock Heegaard {F}loer invariants of {L}egendrian knots in contact
  three-manifolds.
\newblock {\em J. Eur. Math. Soc. (JEMS)}, 11(6):1307--1363, 2009.

\bibitem{not}
L.~Ng, P.~Ozsv{\'a}th, and D.~Thurston.
\newblock Transverse knots distinguished by knot {F}loer homology.
\newblock {\em J. Symp. Geom.}, 6(4):461--490, 2008.

\bibitem{ngtransverse}
Lenhard Ng.
\newblock Combinatorial knot contact homology and transverse knots.
\newblock {\em Adv. Math.}, 227(6):2189--2219, 2011.

\bibitem{osz11}
P.~Ozsv{\'a}th and Z.~Szab{\'o}.
\newblock Knot {F}loer homology, genus bounds, and mutation.
\newblock {\em Topology and its Applications}, 141:59--85, 2004.

\bibitem{oszt}
P.~Ozsv{\'a}th, Z.~Szab{\'o}, and D.~Thurston.
\newblock Legendrian knots, transverse knots, and combinatorial {F}loer
  homology.
\newblock {\em Geom. Topol.}, 12:941--980, 2008.

\bibitem{ost}
Peter Ozsv{\'a}th and Andr{\'a}s~I. Stipsicz.
\newblock Contact surgeries and the transverse invariant in knot {F}loer
  homology.
\newblock {\em J. Inst. Math. Jussieu}, 9(3):601--632, 2010.

\bibitem{ras2}
J.~Rasmussen.
\newblock Knot polynomials and knot homologies.
\newblock {\em Geometry and topology of manifolds}, 47:261--280, 2005.

\bibitem{rutherford}
Dan Rutherford.
\newblock Thurston-{B}ennequin number, {K}auffman polynomial, and ruling
  invariants of a {L}egendrian link: the {F}uchs conjecture and beyond.
\newblock {\em Int. Math. Res. Not.}, pages Art. ID 78591, 15, 2006.

\bibitem{sahamie}
Bijan Sahamie.
\newblock Dehn twists in {H}eegaard {F}loer homology.
\newblock {\em Algebr. Geom. Topol.}, 10(1):465--524, 2010.

\bibitem{sivek}
Steven Sivek.
\newblock Monopole {F}loer homology and {L}egendrian knots.
\newblock {\em Geom. Topol.}, 16:751--779, 2012.

\bibitem{starkston}
Laura Starkston.
\newblock The {K}hovanov homology of {$(p,-p,q)$} pretzel knots.
\newblock {\em J. Knot Theory Ramifications}, 21(5):1250056, 14, 2012.

\bibitem{sv}
Andr{\'a}s~I. Stipsicz and Vera V{\'e}rtesi.
\newblock On invariants for {L}egendrian knots.
\newblock {\em Pacific J. Math.}, 239(1):157--177, 2009.

\bibitem{taubes1}
C.~H. Taubes.
\newblock {Embedded contact homology and Seiberg-Witten Floer cohomology, I}.
\newblock {\em Geom. Topol.}, 14:2497--2581, 2010.

\bibitem{taubes2}
C.~H. Taubes.
\newblock {Embedded contact homology and Seiberg-Witten Floer cohomology, II}.
\newblock {\em Geom. Topol.}, 14:2583--2720, 2010.

\bibitem{taubes3}
C.~H. Taubes.
\newblock {Embedded contact homology and Seiberg-Witten Floer cohomology, III}.
\newblock {\em Geom. Topol.}, 14:2721--2817, 2010.

\bibitem{taubes4}
C.~H. Taubes.
\newblock {Embedded contact homology and Seiberg-Witten Floer cohomology, IV}.
\newblock {\em Geom. Topol.}, 14:2819--2960, 2010.

\bibitem{taubes5}
C.~H. Taubes.
\newblock {Embedded contact homology and Seiberg-Witten Floer cohomology, V}.
\newblock {\em Geom. Topol.}, 14:2961--3000, 2010.

\end{thebibliography}

\end{document}